\documentclass[12pt,a4paper]{amsart}
\usepackage[utf8]{inputenc}
\usepackage[T1]{fontenc}
\usepackage{amsmath,amssymb,amsthm,amsfonts,amsopn}

\usepackage[numeric, abbrev, nobysame]{amsrefs}

\usepackage{mathtools}
\usepackage{physics}
\usepackage{tikz-cd}
\usepackage{booktabs}
\usepackage{array}
\usepackage{cancel}
\usepackage{geometry}
\usepackage{hyperref}

\newcounter{theorem}

\newtheorem{thm}[theorem]{Theorem}
\newtheorem*{thm*}{Theorem}

\newtheorem{prop}[theorem]{Proposition}

\newtheorem*{lemma*}{Lemma}
\newtheorem*{prop*}{Proposition}
\newtheorem*{cor*}{Corollary}

\theoremstyle{definition}

\newtheorem*{definition*}{Definition}

\newtheorem*{example*}{Example}
\theoremstyle{remark}

\newtheorem*{remark*}{Remark}

\newcommand{\ee}{\mathrm{e}}

\newcommand{\DD}{\mathrm{D}}

\newcommand{\CC}{\mathbb{C}}

\newcommand{\RR}{\mathbb{R}}

\DeclareMathOperator{\Alt}{Alt}

\DeclareMathOperator{\im}{im}

\DeclareMathOperator{\rgrad}{grad}
\DeclareMathOperator{\rcurl}{curl}

\let\phi\varphi

\title[The BGG construction]{The Bernstein-Gelfand-Gelfand (BGG) Construction:\\ Algebra, Geometry, and Analysis\\ Part I}
\author{Andreas \v Cap}
\date{June 2026}

\begin{document}

\maketitle

These are rough lecture notes for the first part of the short course ``The Bernstein-Gelfand-Gelfand (BGG) Construction: Algebra, Geometry, and Analysis'' that I gave in the framework of the ESI thematic program ``\href{https://www.esi.ac.at/events/e585/}{Differential Complexes: Theory, Discretization, and Applications}''. Recordings of the lectures are available via the \href{https://www.youtube.com/playlist?list=PLjvY6sC_0lhLXZlh5s0W8KReeCLlxYIV1}{youtube channel of the ESI}. Sincere thanks go to Yuechen Zhu, who, with a little help from his friend Claude, created a first version of this text from pictures of the blackboards taken during the talk and gave it a first editing. 

Some initial remarks
\begin{itemize}
    \item My talk will not address everything that goes under the name BGG in the applied math community. I will only discuss the constructions that have a background in representation theory. 
    \item I will focus on general constructions and not on individual examples (for which the general arguments often would be more complicated than necessary). 
    \item I will mainly work in a smooth setting, functional analytic aspects are only addressed in remarks, the question of discretization will not be touched at all. 
\end{itemize}

\section{Differential forms on $\RR^n$}\label{1}

We will start with a formal approach to differential forms on open domains in $\RR^n$. While not my favorite approach, it is correct and should be not to far to versions inspired from vector calculus. At the same point, I will try to avoid interpretations that I would view as being questionable or misleading.   

\subsection{Setup and notation}

For an index set $I = \{i_1 < \cdots < i_k\} \subset \{1, \ldots, n\}$, introduce a symbol $\dd x^I$ with the convention that $\dd x^\emptyset$ is identified with the constant function $1$. For an open subset $U\subset\RR^n$, a \textit{differential form} on $U$ is a (formal) sum $\omega=\sum_{I}\omega_I\dd x^I$, where the $\omega_I:U\to \RR$ are smooth functions. We will mainly consider the case that in such a sum there are only terms for which $|I|=k$ for some $0\leq k\leq n$, in which we call $\omega$ a \textit{$k$-form} on $U$. The space of $k$-forms is denoted by $\Omega^k(U)$, and $\Omega^*(U) \coloneqq \bigoplus_{k=0}^n \Omega^k(U)$. Each $\Omega^k(U)$ is a vector space under obvious operations. Moreover, $k$-forms can be multiplied by smooth functions via $f\omega = \sum_I (f\omega_I) \, \dd x^I$ for $f\in C^\infty(U,\RR)$ and $\omega=\sum_I\omega_I\dd x^I$.

\subsection{The wedge product}

The wedge product on basis elements is defined by 
\[
\dd x^I \wedge \dd x^J \coloneqq
\begin{cases}
0, & I \cap J \neq \emptyset, \\
s(I, J) \, \dd x^{I \cup J}, & \text{otherwise}.
\end{cases}
\]
Here $s(I, J) \in \{+1, -1\}$ is the sign of the permutation that brings the concatenation $(i_1,\dots,i_k,j_1,\dots j_\ell)$ into increasing order. The wedge product of general differential forms is then defined by 

\[
\omega \wedge \tau \coloneqq \sum_{I, J} (\omega_I \tau_J) \, \dd x^I \wedge \dd x^J.
\]
Of course, working this out explicitly needs a bit of work, since only terms with $I\cap J\neq\emptyset$ provide a non-zero contribution and $\dd x^K$ can be written as $\dd x^I\wedge\dd x^J$ in different ways. Hence we have defined an operation $\wedge \colon \Omega^*(U) \times \Omega^*(U) \to \Omega^*(U)$. One immediately verifies that this is associative on the $\dd x^I$, which immediately implies associativity on general forms. Note that the definition implies that for $I=\{i_1<i_2<\dots<i_k\}$, we get $\dd x^I=\dd x^{i_1}\wedge\dots\wedge\dd x^{i_k}$. 

A moment of thought further shows that for $|I|=k$ and $|J|=\ell$, we get $\dd x^J\wedge\dd x^I=(-1)^{k\ell}\dd x^I\wedge\dd x^J$, so in particular $\dd x^i\wedge\dd x^j=-\dd x^j\wedge\dd x^i$. This immediately extends to general forms, so for $\omega \in \Omega^k(U)$ and $\tau \in \Omega^\ell(U)$ we get 
\[
\tau \wedge \omega = (-1)^{k\ell} \, \omega \wedge \tau.
\]
This property is referred to as \textit{graded commutativity}. 

\subsection{The Exterior Derivative}

For $f \in C^\infty(U, \RR)$, define $\dd f\in\Omega^1(U)$ by $\dd f \coloneqq \sum_{i=1}^n \frac{\partial f}{\partial x^i} \, \dd x^i$. For a general form $\omega = \sum_I \omega_I \, \dd x^I$, we then define 
\begin{equation}\label{d-def} 
   \dd \omega \coloneqq \sum_I \dd \omega_I \wedge \dd x^I = \sum_{I, j} \frac{\partial \omega_I}{\partial x^j} \, \dd x^j \wedge \dd x^I. 
\end{equation}

Note in particular that this definition implies $d(\dd x^I)=0$ for any $I$ and that $d(\Omega^k(U))\subset\Omega^{k+1}(U)$. Expanding this, one again has to take into account that only terms with $j\notin I$ give a non-zero contribution and that $I\cup \{j\}$ can be obtained in several ways. See below for an explicit example.   

From this definition, the key properties of $d$ can be obtained quickly. We first observe that for $f\in C^\infty(U,\RR)$ we can expand the definition as 
$$d(df)=\sum_{i,j}\frac{\partial^2f}{\partial x^j\partial x^i}\dd x^i\wedge \dd x^j=\sum_{i<j}(\frac{\partial^2f}{\partial x^j\partial x^i}-\frac{\partial^2f}{\partial x^i\partial x^j})\dd x^i\wedge\dd x^j=0,$$
by symmetry of second partials and skew symmetry of $\wedge$. 

Next, we analyze compatibility with multiplication by functions:
\[
\begin{aligned}
\dd(f\omega) &= \sum_{I,j} \frac{\partial(f\omega_I)}{\partial x^j} \, \dd x^j \wedge \dd x^I\\
&= \sum_{I,j} \frac{\partial f}{\partial x^j} \omega_I \, \dd x^j \wedge \dd x^I + \sum_{I,j} f \frac{\partial \omega_I}{\partial x^j} \, \dd x^j \wedge \dd x^I\\
&= \dd f \wedge \omega + f \, \dd \omega.
\end{aligned}
\]

This readily generalizes to compatibility with the wedge product. For $\omega \in \Omega^k(U)$ and $\tau\in\Omega^*(U)$, we get $\dd (\omega_I\tau_J)=\tau_J\dd \omega_I+\omega_I\dd \tau_j$. Multiplying this by $\dd x^I\wedge\dd x^J$, we obtain $\tau_J\dd \omega_i\wedge\dd x^I\wedge\dd x^J+(-1)^k\omega_I\dd x^I\wedge \dd\tau_J\wedge \dd x_j$. This easily implies the \textit{graded Leibniz rule}
\begin{equation}\label{Leibniz}
    \dd(\omega \wedge \tau) = \dd\omega \wedge \tau + (-1)^k \omega \wedge \dd\tau.
\end{equation}

This in turn shows that $\dd(\dd \omega_I\wedge\dd x^I)=0$, which shows that $\dd\circ\dd=0$ holds in general (\textit{Nilpotency} of $\dd$). 

\subsection{The de Rham complex}

Nilpotency shows that 
\[
0 \longrightarrow \Omega^0(U) \xrightarrow{\dd_0} \Omega^1(U) \xrightarrow{\dd_1} \cdots \xrightarrow{\dd_{n-1}} \Omega^n(U) \longrightarrow 0
\]
is a \textit{complex}, i.e.\ that $\im \dd_{k-1} \subset \ker \dd_k$.

\textbf{Example: $n=3$.} For $U\subset\RR^3$ open, we have $\Omega^0(U)=C^{\infty}(U,\RR)$ and we can identify $\Omega^3(U)$ with $C^\infty(U,\RR)$ by mapping $f:U\to\RR$ to $f\dd x^1\wedge\dd x^2\wedge\dd x^3$. We also identify $\Omega^1(U)$ with $C^{\infty}(U,\RR^3)$ by sending $(f_1,f_2,f_3)$ to $f_1\dd x^1+f_2\dd x^2+f_3\dd x ^3$. Finally, slightly less obvious, we identify $\Omega^2(U)$ with $C^{\infty}(U,\RR^3)$ by sending $(g_1,g_2,g_3)$ to $g_1\dd x^2\wedge\dd x^3-g_2\dd x^1\wedge \dd x^3+g_3\dd x^1\wedge\dd x^2$. Under this identification, the de Rham complex of $U$ becomes 
\[
C^\infty(U, \RR) \xrightarrow{\rgrad} C^\infty(U, \RR^3) \xrightarrow{\rcurl} C^\infty(U, \RR^3) \xrightarrow{\operatorname{div}} C^\infty(U, \RR).
\]
The identities $\rcurl \circ \rgrad = 0$ and $\operatorname{div} \circ \rcurl = 0$ are instances of $\dd^2 = 0$.

To interpret this, we should view the coordinate forms occurring in the identifications as defining a \textit{frame} for differential forms of various degrees. This just means that general forms can be uniquely written as linear combinations of the frame elements in which the coefficients are smooth functions. Using these frames, the operators of vector calculus occur as expressions for $\dd$ as acting on (triples of) functions. What I definitely \textit{not} want to do is for example view $\rgrad (f)(x)$ as a vector which, if non-zero, is perpendicular to the level set of $f$. The point is that this may be expressed in a much better way, which is independent of coordinates, i.e.\ invariant under local diffeomorphisms. In contrast, the interpretation as a normal to the level set is not even invariant under general linear changes of coordinates. As we shall see below, there also are nice frames consisting of non-constant forms, in which the exterior derivative looks completely different and this leads to interesting applications. 

\section{de Rham cohomology}

I next want to make a few remarks on de Rham cohomology. These are by no means intended to provide a serious introduction to the topic. I mainly want to point out connections to other areas of mathematics which might be helpful to follow up for some applications. 

\begin{definition*}
The $k$th \textbf{de Rham cohomology group} is the quotient vector space
\[
H^k(U) \coloneqq \ker(\dd_k) / \im(\dd_{k-1}).
\]
\end{definition*}

In practice, this just means that one considers equivalence classes $[\omega]$ where $\omega\in\Omega^k(U)$ satisfies $\dd \omega=0$ and $[\omega]=[\hat\omega]$ if and only if there is a form $\tau\in\Omega^{k-1}(U)$ such that $\hat\omega=\omega+\dd \tau$. These classes can be added and multiplied by scalars via  $[\omega_1] + [\omega_2] = [\omega_1 + \omega_2]$, $\alpha[\omega] = [\alpha\omega]$ for $\alpha \in \RR$, so they form a vector space $H^k(U)$.

\subsection{Topological significance}
The de Rham theorem shows that the cohomology spaces $H^k(U)$ coincide with the singular cohomology of $U$ with coefficients in $\RR$. Singular cohomology is defined for very general topological spaces, in sufficiently nice situations there are many other ways to compute it, for example via simplicial cohomology, \v Cech cohomology or via CW-decompositions. (I do not think that following up proofs for this fact is very illuminating, the most illuminating approach probably is via sheaf cohomology. It may be worthwhile, though, to get a bit of exposure to the very robust nature of the methods of algebraic topology.) 

In particular, this implies that $H^k(U)$ is a topological invariant of $U$. The main statement however is that it is a \textit{homotopy invariant}, so it is even unchanged under continuous deformations of $U$. For example, a full torus in $\RR^3$ can be continuously deformed to a circle, which together with a bit of knowledge on algebraic topology allows one to immediately read off its cohomology. Similarly, taking the complement of a small ball in a larger ball, the resulting space can be continuously deformed to the boundary sphere and hence is homotopy equivalent to the sphere $S^2$. Again, this immediately allows one to read off the cohomology. 

Apart from several tools to compute cohomologies, algebraic topology also provides many general results. For example, it is a general result that the cohomology spaces of any compact smooth manifold are finite dimensional. There are similar results for Lipchitz domains in $\RR^n$. Finite dimensional cohomology can have immediate analytic consequences (closed range, Poincar\'{e}-type inequalities, etc.). Hence it may be very valuable to have control on the cohomology of a complex. This provides important motivation for the BGG construction, which can be viewed as constructing complexes out of complexes in a way that does not change the cohomology.  

\section{The Rumin complex on domains in $\RR^3$}

This actually is an example of a BGG construction based on representation theory, but coming from a different kind of geometric structure (a contact structure). It should not be compared to other BGG constructions from a technical point of view, but I hope it illustrates very nicely, how a reduction of a complex that does not change its cohomology may be obtained. All this only needs the formal approach to differential forms discussed in Section\ref{1}. The basic approach here is to rewrite the de Rham complex on $U\subset\RR^3$ in terms of functions and triples of functions but using a different frame for differential forms than the one that leads to vector proxies. 

\subsection{The de Rham complex in a different frame}

We start with the $1$-form $\alpha\in\Omega^1(\RR^3)$ defined by 
\[
\alpha \coloneqq \dd x^3 + x^1 \, \dd x^2.
\]
The definition of the exterior derivative readily shows that $\dd\alpha=\dd x^1\wedge\dd x^2$ and definition of the wedge product then shows that 
\[
\alpha \wedge \dd\alpha = \dd \alpha\wedge\alpha= \dd x^1 \wedge \dd x^2 \wedge \dd x^3.
\]
This is nowhere vanishing, which is the defining property of a \textit{contact form}. 

Now we use this form to defined a frame for differential forms on any open domain $U\subset\RR^3$, basically by ``replacing $\dd x^3$ by $\alpha$''. Indeed, one immediately verifies that any $1$-form on $U$ can be uniquely written as $f_1\dd x^1+f_2\dd x^2+f_3\alpha$ for smooth functions $f_1,f_2,f_3:U\to\RR$. As a major example, writing $\partial_i$ for the $i$th partial derivative
\[
\begin{aligned}
\dd f &= \partial_1 f \, \dd x^1 + \partial_2 f \, \dd x^2 + \partial_3 f \, \dd x^3\\
&= \partial_1 f \, \dd x^1 + (\partial_2 f - x_1 \partial_3 f) \, \dd x^2 + \partial_3 f \, \alpha.
\end{aligned}
\]
The combination $\partial_2 f - x_1 \partial_3 f$ will occur frequently in what follows, so we will denote it by $\tilde{\partial}_2 f$ in what follows. So in particular, identifying $\Omega^1(U)$ with $C^\infty(U,\RR^3)$ via our new frame elements the exterior derivative looks like a modified gradient in which $\partial_2$ is replaced by $\tilde{\partial}_2$. 

The result on $1$-forms readily extends to higher degrees, so we conclude that forms on $U$ admit unique expansions with smooth coefficients in terms of the following frame elements:
\[
\begin{array}{c|c|c|c}
\Omega^0 & \Omega^1 & \Omega^2 & \Omega^3 \\ \hline
1 & \begin{array}{c}\dd x^1 \\ \dd x^2 \\ \alpha\end{array} & \begin{array}{c}\dd x^2 \wedge \alpha \\ -\dd x^1 \wedge \alpha \\ \dd\alpha\end{array} & \alpha \wedge \dd\alpha
\end{array}
\]
It is also very easy to compute the exterior derivative in degree two in this picture:
$$
\dd(g_1\dd x^2\wedge\alpha-g_2\dd x^1\wedge\alpha+g_3\dd \alpha)=(\partial_1 g_1+\tilde{\partial}_2g_2+\partial_3 g_3)\alpha\wedge\dd \alpha. 
$$
So again this just looks like the divergence up to the fact that $\partial_2$ is replaced by $\tilde{\partial}_2$.  
Things become much more interesting in degree one. The crucial computation for what follows is that  for $f_3 \alpha \in \Omega^1(U)$ we get:
\[
\dd(f_3 \alpha) = \dd f_3 \wedge \alpha + f_3 \, \dd\alpha = \partial_1 f_3 \, \dd x^1 \wedge \alpha + \tilde{\partial}_2 f_3 \, \dd x^2 \wedge \alpha + f_3 \, \dd\alpha
\]
and the key feature is the undifferentiated $f_3$ in the last summand. More generally, one readily verifies that identifying $\Omega^1(U)$ and $\Omega^2(U)$ with $C^\infty(U,\RR^3)$ with the frames defined above, the exterior derivative in degree $1$ is represented by the matrix 
\[
\begin{pmatrix} 0 & -\partial_3 & \tilde{\partial}_2 \\ \partial_3 & 0 & -\partial_1 \\ -\tilde{\partial}_2 & \partial_1 & 1 \end{pmatrix},
\]
Again this looks similar to a curl, but in addition to changing $\partial_2$ to $\tilde{\partial}_2$, we now have the $1$ in the lower right corner, which creates a point-wise (tensorial) component. Otherwise put, we have obtained the following complex whose cohomology coincides with the de Rham cohomology of $U$:

\begin{equation}\label{deRham}
\begin{tikzcd}[ampersand replacement=\&]
C^\infty(U, \RR) \arrow[r, "{\left(\begin{smallmatrix}\partial_1\\ \tilde{\partial}_2\\ \partial_3\end{smallmatrix}\right)}"] \&
C^\infty(U, \RR^3) \arrow[r, "{\left(\begin{smallmatrix} 0 & -\partial_3 & \tilde{\partial}_2 \\ \partial_3 & 0 & -\partial_1 \\ -\tilde{\partial}_2 & \partial_1 & 1 \end{smallmatrix}\right)}"] \&[5em] C^\infty(U, \RR^3) \arrow[r, "(\partial_1\, \tilde{\partial}_2\, \partial_3)"] \&[2.5em] C^\infty(U, \RR) 
\end{tikzcd}
\end{equation}

\subsection{Reducing to the Rumin complex}
The above formula for $\dd$ in degree one, has two important consequences:
\begin{enumerate}
\item \textbf{Closed 1-forms.} If $\omega\in\Omega^1(U)$ corresponds to $(f_1,f_2,f_3)$ and $\dd\omega = 0$, then $f_3 = \tilde{\partial}_2 f_1 - \partial_1 f_2$.  
\item \textbf{2-form representatives.} For any $\omega \in \Omega^2(U)$, we can find $\tau\in\Omega^1(U)$ such that $\omega+d\tau$ corresponds to a triple of functions with last component zero. Explicitly, if $\omega$ corresponds to $(g_1,g_2,g_3)$, then $\tau=-g_3\alpha$ does the job. 
\end{enumerate}

The first observation suggests that in degree $1$, one might forget the third component, without losing cohomological content. So we are led to considering the map $\pi:C^\infty(U,\RR^3)\to C^\infty(U,\RR^2)$, $\pi(f_1,f_2,f_3):=(f_1,f_2)$. This also suggests to define $D_0:C^\infty(U,\RR)\to C^\infty(U,\RR^2)$ by $D_0(f):=\pi(df)=(\partial_1f,\tilde{\partial}_2f)$. The second observation suggests that in degree two, all the cohomological content may be contained in triples with vanishing third component. This suggests that we consider the map $i:C^\infty(U,\RR^2)\to C^\infty(U,\RR^3)$ defined by $i(g_1,g_2):=(g_1,g_2,0)$. This in turn leads to an operator $D_2:C^\infty(U,\RR^2)\to C^\infty(U,\RR)$ by composing the modified divergence with $i$, so $D_2(g_1,g_2)=\partial_1g_1+\tilde{\partial}_2g_2$.  

This is not yet sufficient to obtain an operator $D_1$ in degree one, but we can make one more observation. In degree one, we can define an operator $L:C^{\infty}(U,\RR^2)\to C^\infty(U,\RR^3)$, which adds the only third component that can lead to a closed form, i.e.\ $L(f_1,f_2):=(f_1,f_2,\tilde{\partial}_2 f_1 - \partial_1 f_2)$. This is the model case of what we will call a \textit{splitting operator} later on. The crucial properties of $L$ are very easy to prove:

\begin{enumerate}
    \item $L$ is an injective linear differential operator of first order.
    \item For $\omega\in\Omega^1(U)\cong C^\infty(U,\RR^3)$ we get 
    $$\omega\in\mathcal{R}(L)\iff \omega=L(\pi(\omega)) \iff \dd\omega\in\mathcal{R}(i).$$ 
\end{enumerate}
The only part of this that is not obvious is the last equivalence. But this follows from above since $f_3=\tilde{\partial}_2 f_1 - \partial_1 f_2$ is equivalent to vanishing of the last component of the exterior derivative of the form $\omega$ corresponding to $(f_1,f_2,f_3)$. This has several immediate consequences: 

\begin{enumerate}
    \item Given $(f_1,f_2)\in C^\infty(U,\RR^2)$, there is a unique element $D_1(f_1,f_2)\in C^\infty(U,\RR^2)$ such that $\dd(L(f_1,f_2))=i(D_1(f_1,f_2))$. 
    \item For $f\in C^\infty(U,\RR)$, we get $0=\dd(\dd f)$ and hence $\dd f=L(\pi(\dd f))=L(D_0(f))$. In particular, $\dd f=0$ if and only if $D_0(f)=0$. Likewise, $i(D_1(D_0(f)))=\dd(L(D_0(f)))=\dd(\dd f)=0$ so $D_1\circ D_0=0$.
    \item $D_2(D_1(\sigma))=\dd(i(D_1(\sigma)))=\dd(\dd L(\sigma))=0$. 
\end{enumerate}
Otherwise put, we have a complex 
$$
C^\infty(U,\RR) \overset{D_0}\longrightarrow C^\infty(U,\RR^2) \overset{D_1}\longrightarrow C^\infty(U,\RR^2) \overset{D_2}\longrightarrow C^\infty(U,\RR), 
$$
and the maps $(\mathrm{id},L,i,\mathrm{id})$ defined a chain map to the de Rham complex viewed as \eqref{deRham}. While this is not needed in what follows, the explicit formula for $D_1$ can be easily computed to be
\[
D_1 = \begin{pmatrix} \tilde{\partial}_2^2 & -\partial_3 - \tilde{\partial}_2 \partial_1 \\ \partial_3 - \partial_1 \tilde{\partial}_2 & \partial_1^2 \end{pmatrix},
\]
We have already seen above that $\dd f=0\iff D_0(f)=0$ which implies that our chain map induces an isomorphism in cohomology in degree zero. In degree one, we observe that $\dd \omega=0$ implies $\omega=L(\pi(\omega))$ and then $i(D_1(\pi(\omega)))=\dd \omega=0$. This shows that in cohomology, we get $[\omega]=[L(\pi(\omega))]$ so the induced map in first cohomology is surjective. To prove injectivity, we take $\sigma\in C^\infty(U,\RR^2)$ such that $D_1(\sigma)=0$ and assume that $L(\sigma)=\dd f$. Then we get $\sigma=\pi(L(\sigma))=\pi(\dd f)=D_0(f)$, and we get an isomorphism in cohomology in degree one.

In degree two, we have already observed that for $\omega\in\Omega^2(U)$, we find $\tau\in\Omega^1(U)$ such that $\hat\omega=\omega+\dd\tau=i(\sigma)$ for some $\sigma\in C^\infty(U,\RR^2)$. If $\dd \omega=0$, then $0=\dd\hat\omega=\dd i(\sigma)=D_2(\sigma)$ and $[\omega]=[\hat\omega]=[i(\sigma)]$ so we get a surjection in cohomology in degree two. On the other hand, let us start from $\sigma$ such that $D_2(\sigma)=0$ and assume that $i(\sigma)=\dd\tau$ for some $\tau\in\Omega^1(U)$. Then we conclude that $\tau=L(\pi(\tau))$ and hence $i(\sigma)=\dd L(\pi(\tau))=i(D_1(\pi(\tau)))$. Thus $\sigma=D_1(\pi(\tau))$ and the map in cohomology is injective.

In degree three, we have to prove that $\mathcal R(\dd)=\mathcal R(D_2)=\mathcal R(\dd\circ i)$ to show that we get an isomorphism in cohomology. But for $\omega\in\Omega^2(U)$, we can find $\tau\in\Omega^1(U)$ such that $\omega+d\tau\in\mathcal{R}(i)$ and writing $\omega+\dd\tau=i(\sigma)$ we get $D_2(\sigma)=\dd(i(\sigma))=d\omega$. This completes the argument that the Rumin complex computes the de Rham cohomology. 

At this point the whole construction may look rather unmotivated and it should really serve mainly to illustrate why getting some tensorial content into the picture leads to reductions of the complex. To understand the Rumin complex more conceptually, the main ingredient actually is the family of subspace in $\RR^3$ defined as $H_x:=\mathcal N(\alpha(x))$ for $x\in U$. This should be viewed as prescribing, for each point $x\in U$, a plane in the tangent space of $U$ at $x$. The family $\{H_x:x\in U\}$ is called a \textit{contact structure} on $U$. The spaces $C^\infty(U,\RR^2)$ that occur in the complex should (in the spirit of the explanation for differential forms below) be viewed differently in degree one and two. In degree one, the interpretation is a assigning to each $x\in U$ a linear map $H_x\to\RR$ (depending smoothly on $x$). This also shows that there is a natural map from $\Omega^1(U)$ to this space defined by restricting the linear maps $\RR^3\to\RR$ defined by a differential form to to planes $H_x$. In degree two, one should view it as the subspace of $\Omega^2(U)$ consisting of those forms $\omega$, for which in each point $x$, $\omega(x):\RR^3\to\RR$ vanishes if both of its entries lie in $H_x$. 

\section{Geometric interpretation of differential forms}

\subsection{Interpretation via multilinear alternating maps}

We start by interpreting $dx^i$ for $i=1,\dots,n$ as the linear map $\RR^n\to\RR$ which maps a vector $v=(v^1,\dots,v^n)$ to its $i$th component $v^i$. Then any linear map $\RR^n\to\RR$ can be uniquely written as $\sum_{i=1}^na_idx^i$ for $a_1,\dots,a_n\in\RR$. Hence we see that a we can interpret a one-form $\omega\in\Omega^1(U)$, $\omega=\sum\omega_idx^i$ as a smooth map $U\to \RR^{n*}:=L(\RR^n,\RR)$, the space of linear maps from $\RR^n$ to $\RR$. Evidently any smooth map of this type in turn defines a $1$-form on $U$. Of course, one can alternatively view a smooth map $\omega:U\to\RR^{n*}$ as a smooth map $U\times\RR^n\to\RR$, which is linear in the second variable. The correspondence is just given by $(x,v)\mapsto \omega(x)(v)$ and we will frequently switch between the two points of view.   

This generalizes to forms of higher degree replacing $\RR^{n*}$ by the space $\Lambda^k\RR^{n*}$ of $k$-linear, alternating maps from $(\RR^n)^k\to\RR$. Such a map associates to $k$-vectors $v_1,\dots,v_k\in\RR^n$ a number $\phi(v_1,\dots,v_k)$ in a way that is linear in each entry and changes sign if two entries are exchanged. We first interpret $\dd x^I$ with $I = \{i_1 < \cdots < i_k\}$ as the map in $\Lambda^k\RR^{n*}$ defined by 
\[
(v_1, \ldots, v_k)\mapsto \det \begin{pmatrix} v_1^{i_1} & \cdots & v_k^{i_1} \\ \vdots & & \vdots \\ v_1^{i_k} & \cdots & v_k^{i_k} \end{pmatrix}.
\]
It is easy to see that these maps for all $I$ with $|I|=k$ form a basis for $\Lambda^k\RR^{n*}$, so as above we see that we can interpret $\omega\in\Omega^k(U)$ as a smooth map $U\to\Lambda^k\RR^{n*}$ and any such map comes from a $k$-form. Similarly as a above, we can equivalently view $\omega$ as defining a smooth map $U\times\RR^n\times\dots\times\RR^n\to\RR$, which is $k$-linear and alternating in the last $k$ arguments. 

\textbf{Example} Consider $\dd f\in\Omega^1(U)$ for a smooth function $f:U\to\RR$. Then the definitions imply that $\dd f$ coincides with the Fr\'{e}chet derivative:
\[
\dd f(x, v) = \sum_i (\partial_i f)(x) \, v_i = \DD f(x) \cdot v.
\]

At a point $x$ where $\dd f(x) \neq 0$, the kernel $\ker (\dd f(x)) \subset \RR^n$ is a hyperplane. By the implicit function theorem, the level set $\{y : f(y) = f(x)\}$ is locally a smooth $(n-1)$-dimensional submanifold, and $\ker(\dd f(x))$ is the tangent hyperplane to this submanifold in the point $x$. This is similar to the gradient as a normal to the level set, but being a submanifold and being tangent to a submanifold in a point is a property preserved by diffeomorphisms (which act on vectors via the derivative). This is not true for being normal to a submanifold (which is only preserved by restrictions of Euclidean motions). 

\subsection{Pullback of differential forms}

The advantage of this geometric picture of differential forms is that it immediately suggests a way act on differential forms with diffeomorphisms and even with general smooth maps. Let $\Phi \colon U \to V$ be smooth, $U \subset \RR^n$, $V \subset \RR^m$. For $\omega \in \Omega^k(V)$, the \textbf{pullback} $\Phi^*\omega \in \Omega^k(U)$ is defined by
\[
(\Phi^*\omega)(x)(v_1, \ldots, v_k) \coloneqq \omega\qty(\Phi(x))\qty(\DD\Phi(x)(v_1), \ldots, \DD\Phi(x)(v_k)).
\]
Here $\DD\Phi(x) \colon \RR^n \to \RR^m$ is the Fr\'{e}chet derivative of $\Phi$ at $x$. On smooth functions, i.e. for $f\in \Omega^0(V)$, this simplifies to $\Phi^* f = f \circ \Phi$. 

Let us consider the special case of the one-forms $dx^i\in\Omega^1(V)$ for $i=1,\dots,m$. We can write $\Phi$ in components as $\Phi = (\Phi^1, \ldots, \Phi^m)$ with each $\Phi^i \in C^\infty(U, \RR)$. Then by definition, $\Phi^*dx^i(x)$ maps a vector $v$ to the $i$th component of $\DD \Phi(x)(v)$ which is exactly $\dd\Phi^i(x)(v)$. Thus we get $\Phi^*dx^i=d\Phi^i$. 

Our next aim is to understand how the operations introduced above are compatible with pullbacks. There is one slightly annoying step that we have to take, namely to express the wedge product in the language of multilinear forms. It turns out that for $\omega\in\Omega^k(U)$, $\tau\in\Omega^\ell(U)$ and $v_1,\dots,v_{k+\ell}\in\RR^n$ we get
\begin{equation}\label{wedge}
    (\omega\wedge\tau)(x)(v_1,\dots,v_{k+\ell})=\frac{1}{k!\ell!}\sum_{\sigma\in\mathfrak{S}_{k+\ell}}\mathrm{sgn}(\sigma)\omega(x)(v_{\sigma_1},\dots,v_{\sigma_k})\cdot\tau(x)(v_{\sigma_{k+1}},\dots,v_{\sigma_{k+\ell}}).
\end{equation}
Here the sum is over all permutations $\sigma$ of the set $\{1,\dots,k+\ell\}$ and one simply multiplies the real numbers produced by the differential forms. From the definitions, we readily conclude that it suffices to verify this for $\omega=dx^I$ and $\tau=dx^J$. This can be sorted out by direct considerations using uniqueness of the determinant function. Once the formula \eqref{wedge} is available it follows immediate from the definition that pullback is compatible with the wedge product, i.e.\ for a smooth function $\Phi$, one gets $\Phi^*(\omega\wedge\tau)=(\Phi^*\omega)\wedge(\Phi^*\tau)$. 

Having this at hand, we can now give a complete proof of \textit{naturality of the exterior derivative} i.e.\ the fact that $d$ commutes with pullbacks. 

\begin{prop*}
     For any smooth map $\Phi \colon U \to V$ and $\omega \in \Omega^k(V)$:
\[
\Phi^*(\dd\omega) = \dd(\Phi^*\omega).
\]
\end{prop*}
\begin{proof}
     We proceed in three steps:

\textbf{Step 1.} On functions, 
$$
(\Phi^*\dd f)(x)(v)=\dd f(\Phi(x))(D\Phi(x)(v)) = \dd(f \circ \Phi)(x)(v) = \dd(\Phi^* f)(x)(v)
$$
by the chain rule.

\textbf{Step 2.} Writing $\Phi=(\Phi^1,\dots,\Phi^n)$, we have observed that $\Phi^*\dd x^i=\dd\Phi^i$, so in particular $\dd(\Phi^*dx^i)=0$. Since $\dd x^I=\dd x^{i_1}\wedge\dots\wedge\dd x^{i_k}$, compatibility with the wedge product then implies $\dd(\Phi^*\dd x^I)=0$ for any $I$. 

\textbf{Step 3.} For a general form $\omega = \sum_I \omega_I\, \dd x^I$:
\[
\begin{aligned}
\dd(\Phi^*\omega) &= \sum_I \dd((\omega_I \circ \Phi)\, \Phi^*\dd x^I) \\
&= \sum_I \dd(\omega_I \circ \Phi) \wedge \Phi^*\dd x^I + \cancelto{0}{(\omega_I \circ \Phi)\,\dd(\Phi^*\dd x^I)} \\
&= \sum_I \Phi^*(\dd\omega_I) \wedge \Phi^*\dd x^I = \Phi^*\qty(\sum_I {\dd\omega_I} \wedge \dd x^I) = \Phi^*(\dd\omega). \qquad \square
\end{aligned}
\]
\end{proof}

We can apply this to the special case the $\Phi$ is a diffeomorphisms between open subsets of $\RR^n$, which we can view as a general change of coordinates. This shows that the formula for $\dd$ is the same in any coordinate system, which is the naturality interpretation of the result. Moreover, the fact that $\Phi^*dx^i=d\Phi^i$ allows for very efficient computations, as the following example of \textit{spherical coordinates} shows. Viewing spherical coordinates as being defined by a map $\Phi(r, \varphi, \vartheta)=(x, y, z)$, we for example get $x=\Phi^1 = r\cos\varphi\cos\vartheta$ and hence
\[
\Phi^*\dd x = \cos\varphi\cos\vartheta\, \dd r - r\sin\varphi\cos\vartheta\, \dd\varphi - r\cos\varphi\sin\vartheta\, \dd\vartheta,
\]
and similarly for $\dd y$, $\dd z$. Compatibility with the wedge product then shows how to get the formulae for $\dd x\wedge\dd y$ and so on. 

\section{Differential forms on smooth manifolds}

The naturality result also allows us to extend what we have done for open subsets of $\RR^n$ so far to smooth manifolds. This does not depend heavily on whether we talk about submanifolds of $\RR^N$ or about abstract smooth manifolds. In either case, for each $x\in M$, there is a natural definition of the \textit{tangent space} $T_xM$ to $M$ at $x$, which is a vector space of dimension $\dim(M)$. In the case of a submanifold, this is a linear subspace of $\RR^N$, on abstract manifolds a different construction is needed. One can then define a $k$-form on $M$ as associating to each $x\in M$ an element of $\Lambda^kT^*_xM$, i.e.\ an alternating $k$-linear map $(T_xM)^k\to\RR$. One also requires the this depends smoothly on $x$, which proceeds via vector fields, i.e.\ maps associating to each $x\in M$ a tangent vector $\xi(x)\in T_xM$. For submanifolds one can simply define smoothness of a vector field as smoothness as a map to $\RR^N$, for abstract manifolds, a more elaborate treatment via making the union $TM$ of all tangent spaces into a smooth manifold is required. Given a $k$-form $\omega$ and (local) vector fields $\xi_1,\dots,\xi_k\in\mathfrak{X}(M)$, one obtains a (local) function $\omega(\xi_1,\dots,\xi_k):M\to\RR$ defined by $\omega(\xi_1,\dots,\xi_k)(x):=\omega(x)(\xi_1(x),\dots,\xi_k(x))$. The form $\omega$ is then defined to be smooth if and only if for $\omega(\xi_1,\dots,\xi_k)$ is smooth for any choice of smooth local vector fields $\xi_i$. 

This then defines spaces $\Omega^k(M)$ of smooth $k$-forms on $M$ which are vector spaces under point-wise operations and there is an obvious multiplication by smooth functions. Also \eqref{wedge} can be used as a definition to extend the wedge product to manifolds. One can return to the formal approach via expressing things in local coordinates. Given an open subset $U\subset M$ and a diffeomorphism $u=(u^1,\dots,u^i):U\to u(U)$ onto an open subset $u(U)\subset\RR^n$, vector fields can be expressed as linear combinations of the \textit{coordinate vector fields} $\frac{\partial}{\partial u^i}$ and extracting the coefficients defines $\dd u^1,\dots,\dd u^n\in\Omega^1(U)$. We then define $dU^I:=du^{i_1}\wedge\dots\wedge du^{i_k}$ and for $\omega\in\Omega^k(M)$ there are unique smooth functions $\omega_I:U\to\RR$ such that $\omega|_U=\sum_I\omega_Idu^I$. Replacing partial derivatives by the action of coordinate vector fields, we can then define $\dd\omega\in\Omega^{k+1}(U)$ as in \eqref{d-def}. Commutativity of $\dd$ with pullbacks implies that any other choice of local coordinates leads to the same result and hence local results fit together to define $d:\Omega^k(M)\to\Omega^{k+1}(M)$. This also shows that the properties we have observed for open subsets of $\RR^n$, like the Leibniz rule in \eqref{Leibniz} and $\dd\circ\dd=0$ continue to hold on smooth manifolds. Also the connection to algebraic topology extends to this setting.  

I just want to comment briefly on the relation to \textit{integration} (which will not play a big role in what follows). An $n$-form $\omega \in \Omega^n(M)$ is locally expressed as $\omega = \omega_{1\ldots n}\, \dd u^{1\ldots n}$. Under a coordinate change $y = \Phi(x)$, the component transforms as $\omega_{1\ldots n}^x = \omega_{1\ldots n}^y \cdot \det(\DD\Phi)$. Up to taking the absolute value of the determinant, this is compatible with the transformation of multiple integrals. The sign issue is handled by introducing orientations, then one obtains an integral of compactly supported $n$-forms which is invariant under orientation preserving diffeomorphisms. For lower degree forms, one can restrict to oriented submanifolds of the right dimension (i.e.\ form the pullback with respect to the inclusion of the submanifolds) and then integrate this restriction.

\section{Vector-valued differential forms and connections}
The construction of BGG complexes used in applied mathematics usually starts from a ``sum of copies of the de Rham complex'' or from a vector valued de Rham complex. Even in the latter point of view (which is better than the former in my opinion) it is hard to see that this involves more than just the exterior derivative. Since this is crucial for getting a motivation for the BGG construction, I'll elaborate on this point a bit. 

The initial definition of \textit{forms with values in a vector space} is unproblematic. Given a finite dimensional vector space $W$, we can consider the spaces $\Lambda^k\RR^{n*}\otimes W$ of alternating, $k$-linear maps $(\RR^n)^k\to W$. Given an open subset $U\subset\RR^n$, we can then define a $W$-valued $k$-form on $U$ as a smooth map  $U\to\Lambda^k\RR^{n*}\otimes W$. The space $\Omega^k(U,W)$ then is a vector space under point-wise operations and such forms can again be multiplied by $\RR$-valued smooth functions. Choosing a basis $\{w_a\}$ of $W$, we can write a form $\phi\in\Omega^k(U,W)$ as $\phi=\sum_a \phi^a w_a$ with each $\phi_a\in\Omega^k(U)$. (This just means that for $x\in U$ and $v_1,\dots,v_k\in\RR^n$, $\phi^a(x)(v_1,\dots,v_k)$ is the coefficient of $w_a$ in the expansion of the vector $\phi(x)(v_1,\dots,v_k)\in W$ in terms of the basis $\{w_a\}$. This shows that $\Omega^k(U,W)\cong\Omega^k(U)\otimes W$, so we will also write the decomposition as $\phi=\sum_a \phi^a\otimes w_a$. This also readily leads to an expansion $\phi = \sum_{I, a} \phi^a_{I}\, \dd x^I \otimes w_a$. On these forms one then considers the component-wise exterior derivative, i.e.
\[
\dd\phi = (\dd\phi^a)\otimes w_a = \sum_{I, a} ({\dd\phi^a_{I}} \wedge \dd x^I) \otimes w_a.
\]
This definition readily implies that $\dd\circ \dd=0$ also in this case, so we again get a complex. Also, the cohomology of this complex in degree $k$ obvious equals $H^k(U)\otimes W$.  

The additional ingredient referred to above is the distinguished role of the constant forms $\dd x^I\otimes w_a$, in particular the constant functions $w_a$. This also involves a choice of basis, but this is a minor issue. As we shall see below, $W$ usually will be closely related to $\RR^n$, so one should think about the $w_a$ themselves rather as constant vector fields or constant differential forms. This already shows that extensions to manifolds may become tricky, since these concepts are not even defined for manifolds.

\subsection{Connections}

To deal with these issues conceptually, we should actually think of functions $U\to W$ as sections of the trivial vector bundle $U\times W\to U$. This trivial bundle comes with a flat connection, for which the parallel sections are exactly the constant functions. To introduce these concepts more formally, we have to move to a slightly different point of view on $W$-valued differential forms, which puts less emphasis on the role of constant forms. Given a k-Form $\tau\in\Omega^k(U)$ and a smooth function $F:U\to W$, we can define $\tau\otimes F\in\Omega^k(U,W)$ by $(\tau\otimes F)(x)(v_1,\dots,v_k):=\tau(x)(v_1,\dots,v_k)\cdot F(x)$, so one multiplies the vector $F(x)$ by the real number produced by $\tau$. The expansion $\omega=\sum_a\omega^a\otimes w_a$ shows that any $\omega\in\Omega^k(U,W)$ can be written as a finite sum of terms of the form $\tau\otimes F$. Of course, in this more general tensor product notation, there is quite a lot of ambiguity. In particular, for a smooth function $f:U\to\RR$, we get $f(\tau\otimes F)=(f\tau)\otimes F=\tau\otimes (fF)$ (and technically, we are dealing with a tensor product over the algebra $C^\infty(U,\RR)$ here). Using this language, it is now easy to introduce the general concept of a linear connection in this special case. 

\begin{definition*}
     A \textit{linear connection} on the trivial vector bundle $U \times W\to U$ is a linear operator 
\[
\DD \colon C^\infty(U, W) \to \Omega^1(U, W)
\]
that satisfies the Leibniz rule: for $f \in C^\infty(U, \RR)$ and $g \in C^\infty(U, W)$,
\[
\DD(fg) = {\dd f} \otimes g + f\, \DD(g).
\]
\end{definition*}
 
 This Leibniz rule is the differential geometer's way of saying that $\DD$ should be a first order differential operator whose symbol is the identity map on $\RR^{n*}\otimes W$. The component-wise exterior derivative indeed satisfies this definition: Writing $g=\sum_a g^aw_a$ we get $\dd g=\sum_a \dd g^a\otimes w_a$ and 
$$\dd(fg)=\sum_a (g^a\dd f\otimes w_a+f\dd g^a\otimes w_a=df\otimes g+f\dd(g).$$
The constant functions are exactly those functions for which $\dd g=0$. In this context, they are referred to as ``parallel sections''. This justifies the terminology that $\dd$ is the linear connection obtained from the trivialization of the bundle $U\times W\to U$. 

The definition of a linear connection easily implies that the difference two linear connections $\DD$ and $\hat{\DD}$ on the bundle $U\times W$ is a tensorial (point-wise) operation. This can for example be expressed as the fact that there are 
smooth maps $A_i \colon U \to L(W, W)$ such that 
\[
\hat{\DD}(g)(x) - \DD(g)(x) = \sum_i \dd x^i \otimes A_i(x)(g(x)),
\]
and this holds on general vector bundles. In the case of the trivial bundle $U\times W$, one can apply this with $\DD=\dd$, thus relating a general connection to the component wise exterior derivative. On general vector bundles, this can only be done in local trivializations, leading to the usual local description of linear connections via connection 1-forms. In our case this is a bit misleading since it brings back in the distinguished role of constant functions. 

There is a general way to extend a linear connection $\DD$ on a vector bundle $E\to M$ to an operation on $E$-valued differential forms, called the \textit{covariant exterior derivative}. In the picture of the trivial bundle $U\times W$, this amounts to extending $\DD:C^\infty(U,W)\to\Omega^1(U,W)$ to an operation $d^{\DD}:\Omega^k(U,W)\to\Omega^{k+1}(U,W)$ for each $k$. In order to understand this independently of constant functions, we consider $\phi\otimes g$ for $\phi\in\Omega^k(U)$ and $g:U\to W$, so we know that for $f\in C^\infty(U,\RR)$, we have $f\phi\otimes g=\phi\otimes fg$. Then we define
\[
\dd^{\DD}(\phi \otimes g) = {\dd\phi} \otimes g + (-1)^k\, \phi \wedge \DD(g), 
\]
with he wedge product in the second summand defined as 
$$
(\phi\wedge \DD(g))(\xi_1,\dots,\xi_{k+1})=\tfrac{1}{k!}\sum_{\sigma\in\mathfrak{S}_{k+1}}\mathrm{sgn}(\sigma)\phi(\xi_{\sigma_1},\dots,\xi_{\sigma_k})\DD(g)(\xi_{\sigma_{k+1}}).
$$
The the Leibniz rules for $\dd$ and $\DD$ easily imply that this is well defined, i.e.~compatible with the ambiguity $f\phi\otimes g=\phi\otimes fg$, and then we extend it by linearity to general elements of $\Omega^k(U,W)$. Starting from the trivial connection $\DD = \dd$, this recovers the component-wise exterior derivative.

\subsection{Curvature and flat connections}
To understand ``how far'' a general connection $\DD$ on $U\times W$ is from the trivial connection $\dd$, it is natural to ask the question of existence of parallel sections, i.e.\ of functions $g:U\to W$ such that $\DD(g)=0$. First observe that there cannot be many parallel sections. Indeed, from the relation to $\dd$ discussed above, we see that $\DD(g)=0$ is a linear system of first order PDE. Along a smooth curve, this reduces to a linear system of first order ODE, so if $U$ is connected, any solution of the equation $\DD(g)=0$ is uniquely determined by its value in a single point $x_0\in U$. Now suppose that we have $N$ linearly independent solutions $\{g_a\}$ to the equation $\DD(g)=0$ where $N=\dim(W)$. Then the elements $g_a(x_0)\in W$ have to be linearly independent and thus form a basis of $W$, and by uniqueness of solutions, the same has to be true in any point $x\in U$. But this in turn means that any smooth function $g:U\to W$ can be uniquely written as $g=\sum_a f^ag_a$ for smooth functions $f_a:U\to\RR$. In the language introduced above, the sections $g_a$ form a (global) frame for the trivial vector bundle $U\times W$. More generally, any form $\omega\in\Omega^k(U,W)$ can be uniquely written as $\sum_a\phi^a\otimes g_a$ for $k$-forms $\phi^a\in\Omega^k(U)$. And then the definition of $\dd^{\DD}$ readily implies that $\dd^\DD(\sum_a\phi^a\otimes g_a)=\sum_a\dd^\DD(\phi^a\otimes g_a)=\sum_a(\dd\phi^a)\otimes g_a$. So if there is a maximal family of parallel sections, we are just dealing with the component-wise exterior derivative ``in disguise''. 

To understand the question of existence of solutions of $\DD(g)=0$, we return to the point of view of a system of first order PDE. Writing such a system in local coordinates, the natural move is to check for integrability conditions. This means that one takes the equations, applies a partial derivative to them and, if necessary, inserts for first partials that occur in the result from the original equations. Then one check whether symmetry of the second partial derivative imposes additional equations or not.  

In the case of the system $\DD(g)=0$, there is no need to do this by hand, since the covariant exterior derivative does the job. If $\DD(g)=0$, then of course we must have $\dd^{\DD}(\DD(g))=0$. The latter equation turns out to be easier than the former, since for $f\in C^\infty(U,\RR)$ and any $g:U\to W$, the definitions easily imply that $\dd^{\DD}(f\omega)=\dd f\wedge\omega+f\dd^{\DD}(\omega)$ which in turn implies  
\begin{align*}
 \dd^{\DD}(\DD(fg))&=\dd^{\DD}(\dd f\otimes g+f\DD(g))\\
                  &=\dd(\dd f)\otimes g-\dd f\wedge \DD(g)+\dd f\wedge\DD(g)+f\dd^{\DD}(\DD(g))=f\dd^{\DD}(\DD(g)).   
\end{align*}
  
This implies that, while the map $\dd^{\DD} \circ \DD:C^\infty(U,W)\to\Omega^2(U,W)$ is non-zero in general, it is always a point-wise operation: There exist smooth maps $R_{ij} \colon U \to L(W, W)$, such that
\[
\dd^{\DD}(\DD(g)) = \sum_{i < j} \dd x^i \wedge \dd x^j \otimes R_{ij}(x)(g(x)).
\]
So actually, we can view the maps $R_{ij}$ to define an element of $\Omega^2(U,L(W,W))$ and this is an equivalent definition of the curvature of the connection $\DD$.

\begin{definition*}
    The connection $\nabla$ is \emph{flat} if $R = 0$, equivalently $R_{ij} = 0$ for all $i, j$.
\end{definition*} 

Consequences of flatness:
\begin{itemize}
\item Locally, there exists a $(\dim W)$-dimensional family $\{g_\ell\}$ of sections $g_\ell \in C^\infty(U, W)$, which are parallel for $\DD$, i.e.\ satisfy  $\DD(g_\ell) = 0$, such that $\{g_\ell(x)\}$ is a basis of $W$ for every $x$.  Expanding forms in $\Omega^k((U,W)$ in this frame, i.e.\ as $\sum_\ell \omega^\ell\otimes g_\ell$, $\dd^{\DD}$ reduces to component-wise $\dd$. (This follows from the Frobenius theorem applied to the horizontal distribution on the frame bundle.)
\item $\dd^{\DD} \circ \dd^{\DD} = 0$ in every degree, so $(\Omega^*(U, W), \dd^{\DD})$ is a complex, which is called the \emph{twisted de Rham complex}. Local smooth sections which are parallel for $\DD$ define a sheaf and the twisted de Rham complex computes the sheaf cohomology of this sheaf. Hence there are several other tools that can be used to analyze this cohomology,  for example Mayer--Vietoris arguments and \v Cech cohomology.
\end{itemize}

\subsection{Modifying connections}
We have already noted that the difference $\hat{\DD} - \DD$ of two connections is tensorial, so there is $A \in C^\infty\qty\big(U, \RR^{n*} \otimes L(W, W))$ such that $(\hat{\DD} - \DD)(g)(x) = \sum_i \dd x^i \otimes A_i(x)(g(x))$. Conversely, given a connection $\DD$ and $A$, one can use this equation to define a connection $\hat{\DD}$. The difference $\hat{R} - R$ of the curvatures of $\hat{\DD}$ and $\DD$ is then given by applying a non-linear first order differential operator to $A$. Given $\DD$ we can thus ask whether we can modify it to a flat connection. This is a very difficult question in general, and several famous non-linear PDE can be equivalently formulated as flatness of a connection constructed in some way. Somewhat easier, we can start from the trivial (flat) connection $\dd$, and ask for modifications that don't destroy flatness. Still this is a very hard question in general, but we can further simplify things by requiring that $A$ is constant and hence given by an element of $\RR^{n*} \otimes L(W,W)$. In this case, the result is rather easy.

\begin{prop}\label{flat}
 For a constant function $A \in \RR^{n*}\otimes L(W,W)$, the linear connection $\DD = {\dd} + A$ is flat if and only if
\[
A(v_1) \circ A(v_2) = A(v_2) \circ A(v_1) \qquad \text{for all } v_1, v_2 \in \RR^n,
\]   
\end{prop}
Thus to construct flat modifications of $\dd$ by constant functions, we need commuting families of linear maps $W\to W$. 

\subsection{Remarks on the literature}
It is difficult to refer to the literature for the material discussed in this section. Linear connections are often not treated in introductory texts on differential geometry. Of course, Riemannian geometry needs the Levi-Civita connection, but this arises as a derived structure canonically associated to a Riemannian (or pseudo-Riemannian) metric. Here we rather need linear connections as a flexible structure that can be easily deformed. One difficulty is that meaningful general treatments of linear connections have to involve vector bundles. Moreover, in more advanced textbooks like \cite{KN}, there often is a focus on the relation to principal bundles and principal connections, which is beyond what is needed in the current context. What we need here can be found in the appropriate advanced literature on differential geometry as the local study of linear connections, which compares a given connection to the flat connection induced by a local trivialization of the vector bundle. 

\section{The BGG construction}
This section mainly follows my joint article \cite{Cap-Hu} with Kaibo Hu. 

The \textbf{Bernstein--Gelfand--Gelfand (BGG) construction} produces complexes of differential operators from a twisted de Rham complex in two steps:
\begin{enumerate}
\item \textbf{Twist:} For a carefully chosen vector space $W$, modify the component-wise exterior derivative $\dd$ on $\Omega^*(U, W)$ ``artificially'' by a tensorial map $S$ in such a way that one again obtains a complex (and still has control on the cohomology). What one requires is that  $S \circ S = 0$ and $S \circ \dd = -{\dd} \circ S$, which ensures that $\dd_W: = {\dd} + S$ satisfies $\dd_W^2 = 0$. In many cases the complex $(\Omega^*(U, W), \dd_W)$ computes the same cohomology as $(\Omega^*(U, W), \dd)$ and thus has cohomology $H^*(U) \otimes W$.
\item \textbf{Project:} Identify appropriate subspaces in the spaces $\Omega^*(U, W)$ and ``compress'' $\dd_W$ to higher order operators on these subspaces, basically by removing information that cannot influence the cohomology as in the case of the Rumin complex. This leads to a complex $(\Upsilon_*, \mathcal{D})$ which computes the same cohomology as $(\Omega^*(U,W),\dd_W)$. 
\end{enumerate}

\textbf{Remark (Homotopy-operator variant).} In the situation of open subsets of $\RR^n$, there is a simple and popular way to prove the cohomological equivalence in Step 1 via a homotopy operator $K$ satisfying $\dd K - K\dd = S$. Technically speaking, these show that $S$ is chain homotopic to $0$, which in turn implies that $\dd$ is chain homotopic to $\dd_W$ and hence their cohomologies are isomorphic. Due to the special setup in this situation, one can actually do much better here and construct an isomorphism between the complexes
$(\Omega^*(U, W), \dd)$ and $(\Omega^*(U, W), \dd_W)$. It is important to point out here however, that these $K$ operators are much less natural than the $S$ operators, and in particular, one cannot expect a counterpart to the $K$ operators on smooth manifolds. I'll say more about that below. 

\subsection{Motivating example: the twisted complex leading to the  Hessian complex}\label{ex:Hess}

Take $W = W_0 \oplus W_1 = \RR \oplus \RR^{n*}$, so $\Omega^k(U, W) = \Omega^k(U) \oplus \Omega^k(U, \RR^{n*})$. Taking $dx^1,\dots,dx^n$ as a basis for the space $\RR^{n*}$, we can expand a form $\omega\in\Omega^k(U,\RR^{n*})$ as $\sum_{|I|=k,j}\omega_{I,j}dx^I\otimes dx^j$ for $\omega_{I,j}\in C^\infty(U,\RR)$. Hence there is an obvious maps $\Omega^k(U,\RR^{n*})\to\Omega^{k+1}(U)$, which comes from sending $dx^I\otimes dx^j$ to $dx^j\wedge dx^I$, i.e.
$$
S(\omega)=S\left(\sum_{|I|=k,j}\omega_{I,j}dx^I\otimes dx^j\right):=\sum_{|I|=k,j}\omega_{I,j}dx^j\wedge dx^I.
$$
More naturally, viewing $\Omega^k(U)$ as $C^\infty(U,\Lambda^k\RR^{n*})$, we can view $\Omega^k(U,\RR^{n*})$ as $C^\infty(U,\Lambda^k\RR^{n*}\otimes\RR^{n*})$, and the target space can be viewed as the space of $k+1$-linear maps $(\RR^n)^{k+1}\to\RR$, which are alternating in the first $k$ arguments. Now the complete alternation defines a map $\Lambda^k\RR^{n*}\otimes\RR^{n*}\to\Lambda^{k+1}\RR^{n*}$ and by definition $S$ is, up to an appropriate choice of sign, induced by composition with this map. Thus the $S$ operators here are really of point-wise (tensorial) nature. The pointwise operations are recorded in the following diagram. 

\begin{center}
\begin{tikzcd}
\RR^{n*} & \RR^{n*}\otimes \RR & \Lambda^2\RR^{n*}\otimes \RR & \Lambda^3\RR^{n*}\otimes \RR & \cdots\\
\RR^{n*} \arrow[ur, "\cong"] & \RR^{n*}\otimes \RR^{n*}\arrow[ur, "-\Alt"] & \Lambda^2\RR^{n*}\otimes \RR^{n*}\arrow[ur, "\Alt"] & \Lambda^3\RR^{n*}\otimes \RR^{n*}\arrow[ur, "-\Alt"] & \cdots
\end{tikzcd}
\end{center}

Here the top row is $\Lambda^k\RR^{n*} \otimes W_0$, the bottom row is $\Lambda^k\RR^{n*} \otimes W_1$, and we denote the diagonal arrows which induce the $S$ operators as $s_k=s_{k,1}:\Lambda^k\RR^{n*}\otimes W_1\to \Lambda^{k+1}\RR^{n*}\otimes W_0$. Alternatively, we can also view $s_k$ as a linear map $\Lambda^k\RR^{n*}\otimes W\to\Lambda^{k+1}\RR^n\otimes W$, so viewing the rows as describing vectors, this is the matrix $\left(\begin{smallmatrix} 0 & s_k\\ 0 & 0 \end{smallmatrix}\right)$, which readily shows that in this interpretation we have $s_{k+1}\circ s_k=0$ as required. We will use similar notation for the $S$ operators, i.e.\ alternatively view 
$$
S_k:\Omega^k(U,W_1)=\Omega^k(U,\RR^{n*})\to\Omega^{k+1}(U)=\Omega^{k+1}(U,W_0)
$$ 
as a map $\Omega^k(U,W)\to\Omega^{k+1}(U,W)$ which in matrix form looks as $\left(\begin{smallmatrix} 0 & S_k\\ 0 & 0 \end{smallmatrix}\right)$.

Following the general strategy described above, we define $\dd_W=\dd+S$ so in matrix form this reads as $\left(\begin{smallmatrix} \dd & S_k\\ 0 & \dd \end{smallmatrix}\right)$. So we next have to study the relation between $\dd$ and $\dd_W$. 

\textbf{First approach: $K$ operators.} Define $K:\Omega^k(U,\RR^{n*})\to\Omega^k(U)$ as the operator induced by $dx^I\otimes dx^j\mapsto x^jdx^I$, i.e.\ 
$$
K(\omega)=K(\sum_{|I|=k,j}\omega_{I,j}dx^I\otimes dx^j):=\sum_{I,j}x^j\omega_{I,j}dx^I. 
$$
From this definition, one immediately computes that $\dd K-K\dd=S$, which readily implies that $\dd S=-S\dd$ and together with $S\circ S=0$, we get $\dd_W\circ\dd_W=0$. As indicated above,  $\dd K-K\dd=S$ also immediately implies that $\dd$ and $\dd_W$ compute the same cohomology, but one can do better. Of course $\textrm{id}+K$, i.e.\ the matrix $\left(\begin{smallmatrix} \textrm{id} & K \\ 0 & \textrm{id} \end{smallmatrix}\right)$ defines an isomorphisms $\Omega^k(U,W)\to\Omega^k(U,W)$ for each $k$ (with inverse $\textrm{id}-K$ and in matrix form one immediately computes that 
$$
\begin{pmatrix} \textrm{id} & K \\ 0 & \textrm{id} \end{pmatrix} \begin{pmatrix} \dd & S \\ 0 & \dd \end{pmatrix}=\begin{pmatrix} \dd & 0 \\ 0 & \dd \end{pmatrix} \begin{pmatrix} \textrm{id} & K \\ 0 & \textrm{id} \end{pmatrix}. 
$$
Hence $\textrm{id}+K$ defines an isomorphism $(\Omega^*(U,W),\dd_W)\to (\Omega^*(U,W),\dd)$ of complexes. 

\textbf{Second approach: flat connections.}
In degree $0$, $S_0:\Omega^0(U,\RR^{n*})\to \Omega^1(U)$ is induced by the identity on $\RR^{n*}$. From the discussion of flat connections above, we know that we have to re-interpret this in order to think about flatness. We first have to interpret $S$ as a map $W\to\RR^{n*}\otimes W$. This is a linear map from $W$ to linear maps from $R^n\to W$ and hence a bilinear map $W\times\RR^n\to W$. In this picture, the definition of $S$ reads as $\left(\left(\begin{smallmatrix} t\\ \lambda\end{smallmatrix}\right),v\right)\mapsto \left(\begin{smallmatrix} \lambda(v)\\ 0\end{smallmatrix}\right)$ for $t\in\RR$, $\lambda\in\RR^{n*}$ and $v\in\RR^n$. The we have to re-interpret this is a map from $\RR^n$ to the space of linear maps $W\to W$ by ``taking out the other variable''. But this of course is given by sending $v\in\RR^n$ to the matrix $\left(\begin{smallmatrix} 0 & v \\ 0 & 0\end{smallmatrix}\right)$ and clearly all matrices of this form commute. Hence our discussion from above shows that $\DD \coloneqq {\dd} + S_0$ defines a flat connection on the trivial bundle $U\times W$. 

It turns out that in the general approach via representation theory that we will discuss below, the $S$ operators $S_k$ for $k>0$ are automatically constructed in such a way that $\dd_W=\dd+S$ coincides with the covariant exterior derivative $d^{\DD}$ in all degrees. It thus follows from general arguments that $\dd_W\circ\dd_W=0$ and the resulting complex computes the cohomology of the sheaf of locally parallel sections of $\DD$. 

This also provides a conceptual explanation of how the $K$-operators arise and of why $\dd$ and $\dd_W$ lead to the same cohomologies in our case: From our above discussion, we know that locally there is an $n+1$-dimensional family of parallel sections for $\DD$. Starting from a point $x_0\in U$, we can take a fixed element $(t,\lambda)$ in $W$ and try to find a function $g:U\to W$ such that $\DD(g)=0$ and $g(x_0)=(t,\lambda)$. Doing this for $U=\RR^n$ and $x_0=0$ and writing $g=(f,\phi)$ for $f\in C^\infty(\RR^n,\RR)$ and $\phi\in\Omega^1(\RR^n)$, we have to solve the system $df+\phi=0$ and $d\phi=0$ of equations with initial conditions $f(0)=t$ and $\phi(0)=\lambda$. Of course, this means that $\phi$ has to be the constant $1$-form $\lambda$, while for $\lambda=\sum a_idx^i$ we get $f=t+\sum a_ix^i$, so the solution operator is $\textrm{id}+K$. 

\medskip

\textbf{Discussion:} One important point to learn from this example is that the vector space $W$ that shows up in the construction is not really just a vector space. In the example, a smooth map $U\to W$ should rather be viewed as a pair $(f,\phi)$, where $f:U\to\RR$ is a smooth function and $\phi\in\Omega^1(U)$. Correspondingly, a form in $\Omega^k(U,W)$ should rather be viewed as a pair consisting of a $k$ form and a $k+1$-tensor field that is alternating in the first $k$ entries. So one should think of $W$ as being obtained by some kind of construction from $\RR^n$ and of $U\times W$ as being constructed in a corresponding fashion from the tangent bundle of $U$. This of course is also important when thinking about extensions to manifolds. 

In the same spirit, it is important to carefully distinguish between spaces and dual spaces. The situation dual to the Hessian complex constructed here would start from $W$ with $W_0=\RR^n$ and $W_1=\RR$ and correspondingly functions with values in $W$ should rather be viewed as a pair consisting of a vector field and a function. The basic map $S_0:W_1=\RR\to\RR^{n*}\otimes\RR^n=\RR^{n*}\otimes W_1$ maps $a\in\RR$ to the linear map $a\textrm{id}_{\RR^n}:\RR^n\to\RR^n$.  

\subsection{General axiomatic setup for the twisted complex}

The data needed to start the BGG construction are:
\begin{itemize}
\item A vector space $W$ decomposed as $W = W_0 \oplus W_1 \oplus \cdots \oplus W_N$ into a direct sum of linear subspaces. We also assume that each $W_\ell$ is endowed with an inner product $\langle\ ,\ \rangle$.
\item Linear maps $s_{k,\ell} \colon \Lambda^k\RR^{n*} \otimes W_\ell \to \Lambda^{k+1}\RR^{n*} \otimes W_{\ell-1}$ such that $s_{k+1,\ell-1} \circ s_{k,\ell} = 0$. We can also interpret them as defining maps $s_k \colon \Lambda^k\RR^{n*} \otimes W \to \Lambda^{k+1}\RR^{n*} \otimes W$ such that $s_{k+1} \circ s_k = 0$.
\end{itemize}
Hence $(\Lambda^*\RR^{n*}\otimes W,s_*)$ is a complex of finite dimensional vector spaces and we can look at the cohomology spaces $\mathcal{N}(s_{k,\ell})/\mathcal{R}(s_{k-1,\ell+1})$ in each bidegree $(k,\ell)$. In the examples coming from representation theory, it will turn out that in most bidegrees these cohomologies are trivial. In the setting of open subsets of $\RR^n$, where one has a Riemannian metric at hand, we can replace these cohomologies by harmonic subspaces as follows.

Together with the standard inner product on $\RR^n$, the inner products on the spaces $W_\ell$ induce inner products on each of the spaces $\Lambda^k\RR^{n*} \otimes W_\ell$. These can be used to construct \textit{pseudo-inverses} $t_{k,\ell} \colon \Lambda^k\RR^{n*} \otimes W_{\ell} \to \Lambda^{k-1}\RR^{n*} \otimes W_{\ell+1}$: Interpret $s_{k-1,\ell+1}:\Lambda^{k-1}\RR^{n*}\otimes W_{\ell+1}\to\Lambda^k\RR^{n*}\otimes W_\ell$ as defining a linear isomorphism from the orthocomplement of its kernel to its image. Define $t_{k,\ell}$ as the inverse of this isomorphism on $\mathcal R(s_{k-1,\ell+1})$ and as $0$ on $\mathcal R(s_{k-1,\ell+1})^\perp$. Above, we collect these maps to $t_k:\Lambda^k\RR^{n*}\otimes W\to \Lambda^{k-1}\RR^{n*}\otimes W$ and obtain the basic properties 
\[
\mathcal R(t_k) = \mathcal N(s_{k-1})^\perp, \qquad \mathcal N(t_k) = \mathcal R(s_{k-1})^\perp, \qquad t_k \circ s_{k-1} \circ t_k = t_k, \qquad s_k \circ t_{k+1} \circ s_k = s_k.
\]\
We then define \textit{harmonic spaces} 
$$
h_{k,l} \coloneqq \mathcal{N}(s_{k,l})\cap\mathcal{R}(s_{k-1,\ell+1})^\perp=\mathcal{N}(s_{k,\ell}) \cap \mathcal{N}(t_{k,l})\subset \Lambda^k\RR^{n*} \otimes W_l.
$$
We also define $h_k=\oplus_\ell h_{k,\ell}$. 

Given an open subset $U\subset\RR^n$, we can identify $\Omega^k(U,W_\ell)$ with $C^\infty(U,\Lambda^k\RR^{n*}\otimes W_\ell)$. Hence the maps $s_{k,\ell}$ and $t_{k,\ell}$ induce tensorial operators $S_{k,\ell}:\Omega^k(U,W_{\ell})\to\Omega^{k+1}(U,W_{\ell-1})$ and $T_{k,\ell}:\Omega^k(U,W_{\ell})\to\Omega^{k-1}(U,W_{\ell+1})$. We collect them to operators $S_k$ and $T_k$ and $S$ and $T$ and obtain the properties $S \circ S = 0$, $T \circ T = 0$, $S \circ T \circ S = S$, $T \circ S \circ T = T$.

\textbf{Assumption:} We assume that $S \circ \dd = -{\dd} \circ S$ (which will be automatically satisfied in the examples coming from representation theory). This implies that $\dd_W \coloneqq {\dd} + S$ satisfies $\dd_W\circ \dd_W = 0$, and $(\Omega^*(U, W), \dd_W)$ is a complex. In the examples constructed below, the cohomology of this complex is understood. In many cases, it coincides with $H^*(U) \otimes W$.

The result is usually depicted as a \textit{BGG diagram} with $N+1$ rows corresponding to the summands $W_\ell\subset W$ and $n+1$ columns of the form 
\begin{center}
\begin{tikzcd}
{} & {} & {} & {} \\
{} \arrow[r,"\dd"] & \Omega^k(U, W_\ell) \arrow[r, "\dd"] \arrow[ur, "S_{k,\ell}"] & \Omega^{k+1}(U, W_\ell) \arrow[r,"\dd"]\arrow[ur, "S_{k+1,\ell}"] & {}\\
{}\arrow[r,"\dd"]\arrow[ur, "S_{k-1,\ell+1}"] & \Omega^k(U, W_{\ell+1})\arrow[r, "\dd"]\arrow[ur, "S_{k,\ell+1}"] & \Omega^{k+1}(U, W_{\ell+1})\arrow[r,"\dd"]\arrow[ur, "S_{k+1,\ell+1}"] & {}\\
{} \arrow[ur, "S_{k-1,\ell+1}"] & {} \arrow[ur, "S_{k,\ell+2}"] & {} & {} 
\end{tikzcd}
\end{center}

\subsection{The BGG construction}\label{BGG}

Define $\Upsilon_{k,l} \coloneqq C^\infty(U, h_{k,l}) \subset \Omega^k(U, W_l)$ and put $\Upsilon_k:=\oplus_\ell\Upsilon_{k,\ell}$. These will be the spaces showing up in the BGG complex. The key step for the BGG construction is to define the so-called \textit{splitting operators} $L_k:\Upsilon_k\to \Omega^k(U,W)$. Here one should think about $\Upsilon_k$ as $\mathcal{N}(T_k)/\mathcal R(T_{k+1})$ and the splitting operator constructs a specific representative for a class in this quotient. 

Starting from $\sigma \in \Upsilon_{k,\ell}$ (so $S(\sigma) = 0$ and $T(\sigma) = 0$), one constructs elements $\varphi_j \in \mathcal{R}(T)\cap\Omega^k(U, W_j)$ for $j > \ell$ to obtain a representative $L(\sigma)=\sigma+\varphi_{\ell+1}+\dots+\varphi_N$ for the class defined by $\sigma$ such that $T\dd_W(L(\sigma)) = 0$. This is done recursively and the result can be written in the form of a Neumann series, see Section 1.4 of \cite{Cap-Hu} for details. To understand the construction, the main point is to understand one of these steps: 

Since $S(\sigma)=0$, we see that $\dd_W(\sigma) = \dd{\sigma} + {S(\sigma)} = \dd{\sigma} \in \Omega^{k+1}(U, W_\ell)$. Put 
\[
\varphi_{\ell+1} \coloneqq -T\dd_W(\sigma) \in \Omega^k(U, W_{\ell+1}).
\]
This has the required property for the first step: $\dd_W(\varphi_{\ell+1})=\dd\varphi_{\ell+1}-ST\dd_W(\sigma)$ and the terms lie in $\Omega^{k+1}(U,W_{\ell+1})$ and in $\Omega^{k+1}(U,W_{\ell})$, respectively. Applying $T$ to the second summand, we get $-T\dd_W\sigma$, so $T\dd_W(\sigma+\varphi_{\ell+1})=T\dd\varphi_{\ell+1}\in\Omega^k(U,W_{\ell+1})$. Hence we can define $\varphi_{\ell+2}$ to be the negative of this and iterate the argument. This terminates after finitely many steps since $W$ has only finitely many summands and the result can be expressed as a finite Neumann series. In particular, this shows that $L:\Upsilon_{k,\ell}\to\Omega^k(U,W)$ is a differential operator of order at most $N-\ell+1$. We also observe that $L(\sigma)-\sigma=\varphi_{\ell+1}+\dots+\varphi_N$ lies in $\mathcal R(T)$, so $L(\sigma)\in\mathcal N(T_k)$ indeed represents the class of $\sigma$ in $\mathcal N(T_k)/\mathcal{R}(T_{k+1})$.  

While the splitting operators have higher order in general and explicit formulae get complicated quickly, the properties discussed above provide a simple characterization of $L$ by a first order property, which we formulate as a proposition: 

\begin{prop}\label{split} 
For $\sigma\in\Upsilon_{k,\ell}$, assume that $\omega\in\Omega^k(U,W)$ is a form such that 
\begin{enumerate}
    \item $\omega-\sigma\in \mathcal{R}(T_{k+1})$ (which also implies $\omega\in\mathcal{N}(T_k)$)
    \item $T\dd_W\omega=0$. 
\end{enumerate}
Then $\omega=L(\sigma)\in \Omega^k(U,\oplus_{j\geq\ell}W_j)$. 
\end{prop}
\begin{proof}
     Since also $L(\sigma)$ has these two properties, we conclude that $\tau:=\omega-L(\sigma)$ satisfies $\tau\in\mathcal{R}(T_{k+1})$ and $T\dd_W\tau=0$. Since $\mathcal R(T_{k+1})\cap \Omega^k(U,W_0)=\{0\}$ we suppose that $r\geq 1$ is such that $\tau\in\Omega^k(U,\oplus_{j\geq r}W_j)$. Then by the first property, there is $\tilde\tau\in\Omega^{k+1}(U,\oplus_{j\geq r-1}W_j)$ such that $\tau=T\tilde\tau$ and we get $0=T\dd\tau+TST\tilde\tau$. Now the first summand lies in $\Omega^k(U,\oplus_{j>r}W_j)$, so we see that the component in $\Omega^k(U,W_r)$ of $TST\tilde\tau=T\tilde\tau=\tau$ has to vanish. So we must actually have $\tau\in \Omega^k(U,\oplus_{j\geq r+1}W_j)$, whence we can choose $\tilde\tau\in \Omega^{k+1}(U,\oplus_{j\geq r}W_j)$ and then iterate the argument. After finitely many steps, we conclude that $\tau=0$.
\end{proof}

Having the splitting operators at hand, the definition of the BGG operators becomes very easy. By construction, for $\sigma\in\Upsilon_{k,\ell}$ we get $\dd_WL(\sigma)\in\mathcal N(T_{k+1})$ and we can project orthogonally into $\Upsilon_{k+1}=\mathcal N(T_{k+1})\cap\mathcal{R}(T_{k+2})^\perp$ to define $\mathcal D(\sigma)\in\Upsilon_{k+1}$. This defines the \textit{BGG operator} 
\[
\mathcal{D}=\mathcal{D}_k \colon \Upsilon_{k} \to \Upsilon_{k+1}. 
\]
The construction implies that $\mathcal{D}(\Upsilon_{k,\ell})\subset\oplus_{j\geq\ell}\Upsilon_{k+1,j}$ and if several components in this sum are non-zero, then the BGG operator splits accordingly. The construction easily implies that the order of the component $\Upsilon_{k,\ell}\to\Upsilon_{k+1,r}$ of $\mathcal D_k$ is $r-\ell+1$.    

\subsection{Relation to the twisted complex}
The main properties of the BGG construction can now be proved by exploiting only the properties of the splitting operators proved in Proposition \ref{split}. The main observation here is that for $\sigma\in\Upsilon_k$, the form $\omega=\dd_WL(\sigma)$ satisfies $\omega-\mathcal D(\sigma)\in\mathcal R(T)$ and $\dd_W\omega=0$ and hence $T\dd_W\omega=0$. Thus Proposition \ref{split} shows that 
$$
\dd_WL_k(\sigma)=L_{k+1}\mathcal D(\sigma) \text{\ and hence \ } \dd_W\circ L_k=L_{k+1}\circ\mathcal D.
$$
This readily shows that $L_{k+2}\circ\mathcal D_{k+1}\circ\mathcal D_k=\dd_W\circ L_{k+1}\circ\mathcal D_k=\dd_W\circ\dd_W\circ L_{k}=0$, and since $L_{k+2}$ is obviously injective, we see that $\mathcal D_{k+1}\circ\mathcal D_k=0$ for any $k$, so we obtain the \textit{BGG complex} $(\Upsilon_*,\mathcal{D}_*)$. 
The relationship between the full twisted complex and the BGG complex is summarized in:
\begin{center}
\begin{tikzcd}
{}\arrow[r, "\dd_W"] & \Omega^k(U, W)\arrow[r, "\dd_W"]\arrow[d, phantom, "\rotatebox{90}{$\subset$}"] & \Omega^{k+1}(U, W)\arrow[r, "\dd_W"]\arrow[d, phantom, "\rotatebox{90}{$\subset$}"] & {} \\
{} & \mathcal{N}(T) \arrow[d] & \mathcal{N}(T) \arrow[d] & {} \\
{} & \Upsilon_k\arrow[r, dashed, "\mathcal{D}"] \arrow[ur, dashed]\arrow[bend right=60, swap]{uu}{L} & \Upsilon_{k+1} & {}
\end{tikzcd}
\end{center}

The splitting operator $L$ lifts $\Upsilon_k$ into $\mathcal{N}(T) \subset \Omega^k(U, W)$. Applying $\dd_W$ lands in $\mathcal{N}(T) \subset \Omega^{k+1}(U, W)$, and projecting back to $\Upsilon_{k+1}$ gives $\mathcal{D}$.

The equation $\dd_W\circ L=L\circ\mathcal{D}$ also says that the operators $L$ define a chain map from the BGG complex $(\Upsilon_*,\mathcal{D}_*)$ to the twisted complex $(\Omega^*(U.W),\dd_W)$ and hence there are induced maps between the cohomologies of the two complexes.  

\begin{thm*}
 The operators $L$ induce isomorphisms in cohomology.   
\end{thm*} 
\begin{proof}
    There is one more ingredient that we have to prove. Namely, given any form $\omega\in\Omega^k(U,W)$, there is a form $\tau\in\Omega^{k-1}(U,W)$ such that $\omega-\dd_W\tau\in\mathcal{N}(T)$. (Observe that $\Omega^0(U,W)\subset\mathcal N(T)$, so there is no problem in degree zero.) Similarly as before, this is proved recursively. Let $r\geq 1$ be such that $T\omega\in\Omega^k(U,\oplus_{j\geq r}W_j)$ and let $\tau_r$ be the component of $-T(\omega)$ in $\Omega^k(U,W_r)$. Then $T\dd_W\tau_r=T\dd\tau_r-TST\omega$ and the first summand lies in $\Omega^{k-1}(U,W_{r+1})$ while the second equals $-T\omega$. Thus we see that $T(\omega+\dd_W\tau_r)$ has vanishing component in $\Omega^{k-1}(U,W_r)$ and hence lies in $\Omega^k(U,\oplus_{j\geq r+1}W_j)$. Hence we can define $\tau_{r+1}$ as the negative of the component in $\Omega^{k-1}(U,W_{r+1})$ and then iterate the argument. In finitely many steps we obtain a form $\tau$ as required. 

\textbf{Surjectivity in cohomology.} Take $\omega\in\Omega^k(U,W)$ with $\dd_W\omega=0$ and choose $\tau\in\Omega^{k-1}(U,W)$ such that $\tilde\omega:=\omega+\dd_W\tau\in\mathcal N(T_k)$. Of course $\dd_W\omega=0$ and $[\tilde\omega]=[\omega]$ in cohomology. Denote by $\sigma\in\Upsilon_k$ the projection of $\tilde\omega$ to $\mathcal N(T)\cap\mathcal R(T)^\perp$. Then $\tilde\omega-\sigma\in\mathcal R(T)$ and since $\dd_W\tilde\omega=0$, Proposition \ref{split} shows that $\tilde\omega=L(\sigma)$ and hence $\mathcal D_k(\sigma)=0$. Hence $\sigma$ has a class in the cohomology of the BGG sequence and this class is mapped to $[\tilde\omega]=[\omega]$ in the cohomology of the twisted complex. Hence surjectivity is proved.  

\textbf{Injectivity in cohomology.} Here we have to take $\sigma\in\Upsilon_k$ such that $\mathcal{D}_k(\sigma)=0$ and suppose that $L(\sigma) = \dd_W \omega$ for some $\omega\in\Omega^{k-1}(U,W)$ (since this means that the class of $\sigma$ is mapped to zero by the map in cohomology induced by $L$). From above, we see that without loss of generality, we may assume $\omega\in\mathcal N(T)$ and we denote by $\tau\in\Upsilon_{k-1}$ its projection to $\mathcal{N}(T)\cap\mathcal{R}(T)^\perp$. Since $\dd_W\omega=L(\sigma)\in\mathcal{N}(T)$ we again obtain $\omega=L(\tau)$ and hence $L(\sigma)=\dd_WL(\tau)=L(\mathcal D(\tau))$. Injectivity of $L$ then shows that $\sigma=\mathcal D(\tau)$, so the class of $\sigma$ in cohomology is zero and injectivity follows. 
\end{proof}

\subsection{Example: the elasticity complex}\label{ex:elast}

We discuss only the point-wise operations that are the background for constructing the elasticity complex. Take $W = W_0 \oplus W_1 = \RR^{n*} \oplus \Lambda^2\RR^{n*}$. Hence $\RR^{n*}\otimes W_0$ can be viewed as the space of all bilinear forms $\RR^n\times\RR^n\to\RR$, while $W_1$ is the space of skew symmetric bilinear forms. Hence there is an obvious inclusion $s_0:W_1\hookrightarrow \RR^{n*}\otimes W_1$. For $k\geq 2$, we can view $\Lambda^k\RR^{n*}\otimes W_0$ as the space of $k+1$-linear maps $(\RR^n)^{k+1}\to\RR$, which are alternating in the first $k$ entries. Similarly, $\Lambda^{k-1}\RR^{n*}\otimes W_1$ is the space of $k+1$-linear maps $(\RR^{n*})^{k+1}\to\RR$ which are alternating in the first $k-1$ arguments and in the last two arguments. Hence there is a simple map $s_{k-1}:\Lambda^{k-1}\RR^{n*}\otimes W_0\to\Lambda^k\RR^{n*}\otimes W_1$ given by alternating a multilinear map in the first $k$ entries. This leads to the following picture:
\begin{center}
\begin{tikzcd}
\RR^{n*} & \RR^{n*} \otimes \RR^{n*} & \Lambda^2\RR^{n*} \otimes \RR^{n*} & \Lambda^3\RR^{n*}\otimes \RR^{n*} \\
\Lambda^2\RR^{n*}\arrow[ur, "s_0"] & \RR^{n*} \otimes \Lambda^2\RR^{n*}\arrow[ur, "s_1"] & \Lambda^2\RR^{n*}\otimes \Lambda^2\RR^{n*}\arrow[ur, "s_2"] & {}
\end{tikzcd}
\end{center}

As noted above, $s_0$ is injective and one can verify directly that $s_1$ is a linear isomorphism, while for all $k\geq 2$, the map $s_k$ is surjective. This readily leads to a description of the cohomology spaces $h_{k,\ell}$. Injectivity of $s_0$ implies that $h_0=h_{0,0}=\RR^{n*}$. Since $s_1$ is a linear isomorphism, $h_1=h_{1,0}=(\RR^{n*}\otimes\RR^{n*})/\mathcal R(s_0)\cong S^2R^{n*}$. For $k\geq 2$, we get $h_k=h_{k,1}=\mathcal{N}(s_k)\subset\Lambda^k\RR^{n*}\otimes\Lambda^2\RR^{n*}$. In the actual construction of the twisted complex, one has to add signs to the maps constructed here in order to ensure the relations $\dd\circ S=-S\circ \dd$.

\medskip

At this point we can start understanding why representation theory might be helpful. The maps $s_k$ that occur in this example (as well as those for the Hessian complex) are of a special type. They are particularly natural in the sense that they can be defined without choosing bases. Technically speaking, the spaces we deal with come with natural actions of the group $GL(n,\RR)$ of invertible $n\times n$-matrices and the maps $s_k$ are \textit{equivariant}, i.e.\ compatible with these actions, see below for more precise definitions. 

In these simple examples, representation theory leads to descriptions of the kernels and images of the maps $s_k$, and hence to explicit descriptions of the spaces $h_{k,\ell}$, which are obviously relevant for understanding what the resulting BGG sequences look like. There are more advanced applications of representation theory, which directly lead to
\begin{itemize}
    \item A vector space $W$ with an appropriate decomposition $W=W_0\oplus\dots\oplus W_N$. 
    \item Maps $s_{k,\ell}$ which are $GL(n,\RR)$-equivariant (which is needed to get properties like affine invariance for the subsequent constructions) and satisfy $s_{k+1}\circ s_k=0$. 
    \item A description of the resulting subspaces $h_{k,\ell}$.
\end{itemize}

\section{Some representation theory}
Representation theory is a very general concept for various types of algebraic structures and in particular for groups. For our purpose we only deal with the subgroups $O(n)$ (orthogonal matrices) and $SL(n,\RR)$ (matrices of determinant $1$) of $GL(n,\RR)$ (invertible matrices). For these cases, representation theory is closely related to linear algebra. More precisely, $GL(n,\RR)$ governs linear algebra (of $n$-dimensional spaces) while $O(n)$ governs the linear algebra of inner product spaces (of dimension $n$). The group $SL(n,\RR)$ admits a similar interpretation in terms of vector spaces with a distinguished volume form, but for us the use of $SL(n,\RR)$ is mainly of technical nature, and this interpretation is less important.  

\subsection{Matrix groups and their representations}
Matrix groups provide the most important examples of Lie groups. Their study only needs basic analysis and no methods of differential geometry. The correspondence between matrix groups and their Lie algebras actually provides beautiful examples for the methods of basic analysis. 

Continuity of the determinant function implies that $GL(n,\RR)$ is an open subset of the space $M_n(\RR)$ of $n\times n$-matrices with real entries, which can be identified with $\RR^{n^2}$. Hence all the methods of higher dimensional analysis can be applied to functions $GL(n,\RR)\to\RR^m$ for any $m$. The formal definition of a matrix group is as a subgroup $G\subset GL(n,\RR)$ which is a closed subset of the topological space $GL(n,\RR)$. This is of course satisfied for $G=O(n)$ and $G=SL(n,\RR)$, which even are closed in $M_n(\RR)$ (which is not required in the definition). Indeed, the closedness properties follow immediately from the fact that our groups preserve some ``structure'', namely the inner product respectively a volume form on $\RR^n$, which gives a description of the group in terms of continuous equations. This is a general principle that applies to many symmetry groups. 

There is a general theorem (first proved by J.\ von Neumann) stating that any matrix group $G\subset GL(n,\RR)$ is a smooth submanifold of $M_n(\RR)\cong\RR^{n^2}$. In particular, one can talk about smoothness of maps from $G$ to Euclidean spaces and to other submanifolds and hence in particular of maps between matrix groups. 

A \textit{representation} of a matrix group $G$ on a finite dimensional vector space $V$ is a smooth left action of $G$ on $V$ by linear maps. So we need a smooth map $G\times V\to V$ that we write as $(A,v)\mapsto A\cdot v$ such that $\mathbb I\cdot v=v$ for the identity matrix $\mathbb I$ and such that for any $A,B\in G$, we get $(AB)\cdot v=A\cdot(B\cdot v)$. Alternatively, we can view this as a group homomorphism $G\to GL(V)$, which is smooth in the following sense: For a choice of Basis $v_1,\dots,v_N$ for $V$ on can view the homomorphism as a map $G\to M_N(\RR)$, and this map has to be smooth. For the groups $O(n)$ and $SL(n,\RR)$ the obvious representation is the \textit{standard representation} on $\RR^n$, which is defined by $(A,v)\mapsto Av$ and corresponds to the inclusion into $GL(n.\RR)$. One of the possible interpretations of the action of invertible matrices on $\RR^n$ is to describe base changes, and the restriction elements of $O(n)$ correspond to changes between orthonormal bases. 

For the groups we consider, one should think about representations as being defined on a vector space $V$ obtained from $\RR^n$ by some functorial construction. This implies that any basis of $\RR^n$ gives rise to a basis of $V$, and the action of a matrix $A$ on $V$ describes the effect of a base change on $\RR^n$ on the induced bases of $V$. The standard examples of such constructions are built up from dualization, tensor products and passing to invariant subspaces. Here a linear subspace $E\subset V$ is called \textit{$G$-invariant} if and only if for any $A\in G$ and $v\in E$, we get $A\cdot v\in E$. Of course, we then get a representation of $G$ on $E$.

A basic example, which also illustrates the connection to linear algebra is the space $L(\RR^n,\RR^n)\cong M_n(\RR)$ of linear maps, which can also be viewed as $\RR^n\otimes\RR^{n*}$. The natural representation of $GL(n,\RR)$ on this space is given by $A\cdot X:=AXA^{-1}$ and this is well known to correspond to the effect of a base change in $\RR^n$. Parts of linear algebra like the Jordan normal form (over $\mathbb C$) answer questions of when two matrices are related by the action of a group element, in general this means understanding the orbits of a group action. The questions of representation theory are easier in general, since they are linear in nature. In particular, it turns out that the only $GL(n,\RR)$-invariant subspaces on $L(\RR^n,\RR^n)$ are the $1$-dimensional subspace of multiples of the identity and the complementary subspace of trace-free maps. Here things become more complicated by restricting to the subgroup $O(n)$, since $L(\RR^n,\RR^n)$ decomposes into a direct sum of three invariant subspaces (multiples of the identity, trace-free symmetric maps and skew symmetric maps). 

A related simple example is $V:=\RR^{n*}\otimes\RR^{n*}$, the space of bilinear maps $\RR^n\times\RR^n\to\RR$ with the representation defined by $(A\cdot b)(v_1,v_2):=b(A^{-1}v_1,A^{-1}v_2)$. Indeed it turns out that inverting the matrix in this formula is needed in order to obtain the required action property. From the definition, it follows readily, that the spaces $S^2\RR^{n*}$ of symmetric bilinear forms and $\Lambda^2\RR^{n*}$ of skew-symmetric bilinear forms are $GL(n,\RR)$-invariant subspaces of $V$. As representations of $O(n)$, $\RR^n\otimes\RR^{n*}$ and $\RR^{n*}\otimes\RR^{n*}$ are actually isomorphic, so there is again a decomposition into a direct sum of three invariant subspaces. This corresponds to the fact that over $O(n)$ a symmetric bilinear form can be uniquely decomposed into a multiple of the inner product and a trace-free symmetric form. More generally one can consider tensor products of the form $\otimes^k\RR^{n*}\otimes\otimes^\ell\RR^n$ and invariant subspaces in there defined by symmetry and skew-symmetry properties. In particular, the spaces $\Lambda^k\RR^{n*}$ carry natural representations of $GL(n,\RR)$, $SL(n,\RR)$ and $O(n)$. The general rule for tensor products is that for representations on $V$ and $W$, one obtains a representation on $V\otimes W$ by $g\cdot (v\otimes w):=(g\cdot v)\otimes(g\cdot w)$.

\subsection{Two fundamental results}

Given a representation of $G$ on $V$, we have already met the concept of an invariant subspace $E\subset V$. For any representation, the trivial subspaces $\{0\}\subset V$ and $V\subset V$ are obviously invariant. If these are the only invariant subspaces, the representation $V$ is called \textit{irreducible}.  

Another fundamental concept we need is the one of an equivariant map, which is the right definition of a morphism between representations:

Given representations of $G$ on $V$ and $W$, a linear map $f \colon V \to W$ is called \textit{$G$-equivariant} (or a \textit{intertwining map}) if $f(A \cdot v) = A\cdot f(v)$ for all $v \in V$ and all $A \in G$.   This can be equivalently formulated in terms of the natural action of $G$ on the space $L(V,W)$ of linear maps. Similarly as in the case of $\RR^n$ discussed above, the natural action of $L(V,W)$ is given by $(A\cdot f)(v):=A\cdot(f(A^{-1}\cdot v))$ for any $v\in V$ and $A\in G$. (Note that here the dot denotes three different actions/representations. In the left hand side, it denotes the representation on $L(V,W)$ we define, while in the right hand side, the first dot is the action on $W$ and the second dot the one on $V$.) Hence $f$ is equivariant if and only if $A\cdot f=0$ for any $A\in G$. This property is often phrased as $f$ being a \textit{$G$-invariant} element of $L(V,W)$. 

The first fundamental result for the groups we are interested in is referred to as \textit{complete reducibility} of representations. 

\begin{thm}\label{comp-red}
  (H.\ Weyl). Let $G$ be one of the groups $\mathrm{GL}(n, \RR)$, $\mathrm{SL}(n, \RR)$, or $O(n)$. If $V$ is a finite-dimensional representation of $G$ and $E_1 \subset V$ is a $G$-invariant subspace, then there exists a a complementary $G$-invariant subspace $E_2 \subset V$, i.e.\ we get $V = E_1 \oplus E_2$.
\end{thm} 

For $O(n)$, this result is relatively easy to prove using compactness of the group. Starting from a representation on $V$, one can start from any inner product on $V$ and then average it by integration over $O(n)$ to obtain an \textit{invariant inner product} $\langle\ ,\ \rangle$ on $V$. This means that for $v_1,v_2\in V$ and $A\in O(n)$, one has $\langle A\cdot v_1,A\cdot v_2\rangle=\langle v_1,v_2\rangle$. (For representations constructed from $\RR^n$, one can alternatively construct an $O(n)$-invariant inner product from the standard inner product on $\RR^n$.) Given an $O(n)$-invariant subspace $E_1\subset V$, one can now define $E_2:=(E_1)^\perp$ which of course is complementary to $E_1$. For $v_1\in E_1$, $v_2\in E_2$ and $A\in O(n)$ we get $\langle v_1,A\cdot v_2\rangle=\langle A^{-1}\cdot v_1,A^{-1}\cdot A\cdot v_2\rangle=\langle A^{-1}\cdot v_1,v_2\rangle=0$, since $A^{-1}\cdot v\in E_1$ by invariance. Hence $E_2\subset V$ is invariant, too. 

Representations that posses an invariant inner product are called \textit{unitary} and for such representations complete reducibility holds (by the same argument). For $GL(n,\RR)$ and $SL(n,\RR)$ finite dimensional representations are \textit{never} unitary and complete reducibility has to be proved by other methods. This can be either done by reducing to the case of $O(n)$ (Weyl's unitary trick) or by linear algebra methods. 

Given an invariant subspace $E_1\subset V$ and a resulting decomposition $V=E_1\oplus E_2$, one can iterate this. If $E_1$ has an invariant subspace, then we can decompose it and hence $V$ further and like-wise for $E_2$. This is possible until each of the summands does not admit (non-trivial) invariant subspaces. This proves:

\begin{cor*} 
Every finite-dimensional representation $V$ of $G$ (with $G$ as above) admits a decomposition into a direct sum of invariant subspaces, which are irreducible representations of $G$. 
\end{cor*}

The key consequence of complete reducibility is that to understand general representations, it basically suffices to understand irreducible representations. The second fundamental result we discuss, called \textit{Schur's lemma}, shows that this is a tremendous simplification, since irreducible representations show behavior that is similar to one-dimensional vector spaces. 

\begin{thm}\label{Schur} (Schur's lemma).

(1) Let $E_1, E_2$ be irreducible representations of any group $G$ and let $f \colon E_1 \to E_2$ be equivariant. Then either $f = 0$ or $f$ is an isomorphism.

(2) Let $E$ be a complex irreducible representation of any group $G$ (so $E$ is a complex vector space and $G$ acts by complex linear maps) and $f \colon E \to E$ a $\CC$-linear equivariant map. Then $f = \lambda \, \mathrm{id}$ for some $\lambda \in \CC$.
\end{thm}
\begin{proof}
This is extremely easy to prove. For part 1, we claim that $\mathcal N(f)\subset E_1$ and $\mathcal R(f)\subset E_2$ are $G$-invariant subspaces. If $f(v)=0$, then $f(g\cdot v)=g\cdot f(v)=0$ by linearity of the action. Likewise $g\cdot f(v)=f(g\cdot v)$. But this implies that either $\mathcal{N}(f)=E_1$ and hence $f=0$ or $\mathcal N(E)=\{0\}$ and hence $f$ is injective. In the second case $\mathcal R(f)\neq \{0\}$ and hence $\mathcal R(f)=E_2$ and $f$ is surjective. For part 2, $f$ has to have an eigenvalue and one immediately verifies that the corresponding eigenspace is a non-zero invariant subspace of $E$ and hence has to coincide with $E$. 
\end{proof}

The restriction to complex representations in part 2 is not a problem, one can deal with real representations via complexification. Moreover, the result holds in the real case provided that $f$ has a real eigenvalue. 

The two basic results in Theorems \ref{comp-red} and \ref{Schur} then allow us to analyze equivariant maps between general representations. Given representations $V$ and $W$, we first decompose into irreducibles as $V=E_1\oplus\dots\oplus E_N$, $W=F_1\oplus\dots\oplus F_M$. Given an equivariant map $f:V\to W$, we can restrict to one of the irreducible components $E_i\subset V$ and then project to an irreducible component $F_j\subset W$ to obtain an equivariant map $f_{ij}:E_i\to F_j$. By Schur's lemma, this has to be zero unless $F_j\cong E_i$. So in particular, if none of the components $F_j$ is isomorphic to $E_i$, then $f|_{E_i}=0$ and so on. What one does in practice if fixing representatives $E_i$ of the isomorphism classes of irreducible representations and then taking isomorphic irreducible components together. So one would write $V\cong E_1^{n_1}\oplus\dots\oplus E_K^{n_K}$ and the number $n_i$ is called the \textit{mutltiplicity of the irreducible representation $E_i$ in $V$}. Decomposing $W$ in a similar fashion with multiplicities $m_j$, the basic description of equivariant maps is then as follows: One only has to consider irreducible components $E_i$ that occur in both $V$ and $W$ with non-zero multiplicity. If $E_i$ occurs with multiplicities $n_i$ in $V$ and $m_i$ in $W$, the this components contributes a factor to the space of equivariant maps which is isomorphic to the space $L(\RR^{n_i},\RR^{m_i})$ of matrices of sizes $m_i\times n_i$.  

\subsection{Examples: decomposing tensor products; Young diagrams}
It is important to realize that determining the multiplicities with which irreducible representations occur in a given representation often is an algorithmic procedure that can be carried out by software. Understanding more explicitly, how irreducible components sit in a given representations can be more complicated and finding explicit formulae for projections can become tedious. Representation theory provides a number of tools to help in these problems, but a certain amount of experience and some work may be required to deal with concrete cases. 

Let us start with the space $\RR^{n*}\otimes\RR^{n*}$ of bilinear maps $\RR^n\times\RR^n\to \RR$. Symmetric bilinear maps and alternating bilinear maps evidently form invariant subspaces and one readily verifies that
\begin{equation}\label{tens2}
\RR^{n*} \otimes \RR^{n*} \cong S^2\RR^{n*} \oplus \Lambda^2 \RR^{n*},
\end{equation}
as a representation of $GL(n,\RR)$ (and hence also for the subgroups $SL(n,\RR)$ and $O(n)$). One then shows that for $SL(n,\RR)$ (and thus also for the bigger group $GL(n,\RR)$) both summands are irreducible for $n\geq 2$, which can be done by linear algebra methods. We have already noted above that as a representation of $O(n)$, $S^2\RR^{n*}$ decomposes further into multiples of the inner product and trace-free symmetric bilinear forms. One can also see subtleties coming up for $n=2$. The space $\Lambda^2\RR^{2*}$ is one-dimensional and spanned by the determinant, so for $SL(2,\RR)$ this is isomorphic to to trivial representation on $\RR$ (i.e. $A\cdot t=t$ for all $A\in G$ and $t\in\RR$), while $GL(2,\RR)$ acts non-trivially. 

We will now focus on the case of $\mathrm{GL}(n, \RR)$ (and $\mathrm{SL}(n, \RR)$) for which one can describe isomorphism classes of irreducible representations by Young diagrams (or equivalently by  highest weights). In this notation:
\begin{itemize}
\item $\square$ denotes the representation $\RR^{n*}$,
\item $\begin{array}{l} \square\!\square \end{array}$ denotes $S^2 \RR^{n*}$,
\item $\begin{array}{l} \square \\[-.5em] \square \end{array}$ denotes $\Lambda^2\RR^{n*}$.
\end{itemize}

The tensor product decomposition from \eqref{tens2} reads as:
\[
\square \otimes \square \;=\; \begin{array}{l} \square\!\square \end{array} \;\oplus\; \begin{array}{l} \square \\[-.5em] \square \end{array}
\]

Passing to the triple tensor product $\RR^{n*}\otimes\RR^{n*}\otimes\RR^{n*}$, things get a bit more complicated. Of course, in the space of trilinear maps, we have the subspaces $S^3\RR^{n*}$ and $\Lambda^3\RR^{n*}$ which intersect only in zero. Looking at dimensions, we see that these two subspaces cannot make up the full triple tensor product, so by complete reducibility, there is an invariant subspace $W\subset \RR^{n*}\otimes\RR^{n*}\otimes\RR^{n*}$ such that 
\begin{equation}\label{tens3}
\RR^{n*}\otimes\RR^{n*}\otimes\RR^{n*}=S^3\RR^{n*}\oplus W\oplus\Lambda^3\RR^{n*}.
\end{equation}
It again turns out that $S^3\RR^{n*}$ and $\Lambda^3\RR^{n*}$ are irreducible representations of $SL(n,\RR)$. 
In Young diagram notation, they are denoted by $\begin{array}{l} \square\!\square\!\square \end{array}$, and by $\begin{array}{l} \square \\[-.5em] \square \\[-.5em] \square \end{array}$, respectively. 

We can do a bit better using what we know already by observing that 
\begin{align*}
\RR^{n*}\otimes\RR^{n*}\otimes\RR^{n*}&=(\RR^{n*}\otimes\RR^{n*})\otimes\RR^{n*}=(S^2\RR^{n*}\oplus\Lambda^2\RR^{n*})\otimes\RR^{n*}\\
&=(S^2\RR^{n*}\otimes\RR^{n*})\oplus (\Lambda^2\RR^{n*}\otimes\RR^{n*}). 
\end{align*}
The two summands here consist of trilinear maps which are symmetric respectively alternating in the first two entries. In particular, $S^3\RR^{n*}\subset S^2\RR^{n*}\otimes\RR^{n*}$ and is a proper subspace in there, so there is an invariant complement $W_1$ such that $S^2\RR^{n*}\otimes\RR^{n*}=S^3\RR^{n*}\oplus W_1$. Explicitly, we can realize $W_1$ as the kernel of the complete symmetrization. In the same way, the kernel of the complete alternation is an invariant subspace $W_2\subset\Lambda^2\RR^{n*}\otimes\RR^{n*}$, which is complementary to $\Lambda^3\RR^{n*}$. Hence the subspace $W$ in \eqref{tens3} itself decomposes as $W_1\oplus W_2$. 

Proceeding further needs a bit of work. Indeed, one verifies that $W_1$ and $W_2$ are isomorphic irreducible representations, which (for $n\geq 3$) are not isomorphic to any of the other representations we have met already. The isomorphism class of these representations is denoted in Young diagram notation by $\begin{array}{l} \square\!\square \\[-.5em] \square \end{array}$, so the decompositions from above read as 
$$
\begin{array}{l} \square\!\square\end{array}\otimes \square \cong \begin{array}{l} \square\!\square\!\square \end{array}\oplus \begin{array}{l} \square\!\square \\[-.5em] \square \end{array} \quad \begin{array}{l} \square \\[-.5em] \square \end{array}\otimes\square \cong   \begin{array}{l} \square\!\square \\[-.5em] \square \end{array}\oplus \begin{array}{l} \square \\[-.5em] \square \\[-.5em] \square \end{array}, 
$$
which gives an impression on how the decompositions into irreducibles of tensor products look in more general cases. There are general rules in terms of adding the boxes of one factor to the diagram of the other factor in ``admissible places''.  The decomposition of the triple tensor product thus reads as 
\[
\square \otimes \square \otimes \square \;=\; \begin{array}{l} \square\!\square\!\square \end{array} \;\oplus\; \qty({\begin{array}{l} \square\!\square \\[-.5em] \square \end{array}} \oplus {\begin{array}{l} \square\!\square \\[-.5em] \square \end{array}}) \;\oplus\; \begin{array}{l} \square \\[-.5em] \square \\[-.5em] \square \end{array},
\]
so here we get an irreducible component with multiplicity $2$. Dealing with the corresponding summand needs a bit of care, though. The sum of the two components is a canonical subspace (the joint kernel of the complete symmetrization and the complete alternation). Decomposing it into a sum of two isomorphic representations represents a choice, analogous to a choice of basis in a two dimensional vector space. So this is not just a discrete freedom but there are continuous parameters involved. 

\medskip

Understanding basics about irreducible representations and the decompositions of tensor products into irreducibles, we can sort out the questions about kernels and images of the maps giving rise to the BGG diagrams discussed above. For the Hessian complex, the relevant maps are the complete alternations $s_k:\Lambda^k\RR^{n*}\otimes\RR^{n*}\to\Lambda^{k+1}\RR^{n*}$, so they are equivariant by construction. It turns out that $\Lambda^{k}\RR^{n*}$ is irreducible for any $k$. Hence the fact that $s_k\neq 0$ immediately implies that $s_k$ is surjective for any $k$. Now for $k=0$, both the source and the target is $\RR^{n*}$, so $s_0$ is an isomorphism. For $k>0$, $\Lambda^k\RR^{n*}\otimes\RR^{n*}$ contains $\Lambda^{k+1}\RR^{n*}$ as an invariant subspace and the kernel of the alternation provides an irreducible complementary subspace. For $k=2$, we obtain the representation $\begin{array}{l} \square\!\square \\[-.5em] \square \end{array}$, in higher degrees generalizations of this with $k$ boxes in the first column. 

For the elasticity complex, the maps are $s_k:\Lambda^k\RR^{n*}\otimes\Lambda^2\RR^{n*}\to\Lambda^{k+1}\RR^{n*}\otimes\RR^{n*}$ given by alternating in the first $k+1$ entries. For $k=0$, this is the inclusion of $\Lambda^2\RR^{n*}$ to $\RR^{n*}\otimes\RR^{n*}$, so $s_0$ is injective. For $k=1$ we get $s_1:\RR^{n*}\otimes\Lambda^2\RR^{n*}\to \Lambda^2\RR^{n*}\otimes\RR^n$, so these clearly are isomorphic representations. From above, we know that they split into a direct sum of two irreducible components, and one only has to check that $s_1$ in non-zero on one non-zero element in each of the two components to verify that $s_1$ is an isomorphism. We also know from above that also for $k\geq 2$, the target space $\Lambda^{k+1}\RR^{n*}\otimes\RR^{n*}$ splits into two irreducible components. On the other hand, $\Lambda^k\RR^{n*}\otimes\Lambda^2\RR^{n*}$ splits into three irreducible components. Two of these are the ones that occur in $\Lambda^{k+1}\RR^{n*}\otimes\RR^{n*}$, the third irreducible component is non-isomorphic to these two and hence has to be contained in the kernel of $s_k$. Checking that $s_k$ is non-zero on one element in each of the first two components, one thus obtains that $s_k$ is surjective and its kernel coincides with the third irreducible component mentioned above. 

A similar analysis applies to the other complexes obtained from multiforms. In these cases $W=W_0\oplus W_1$ with $W_0=\Lambda^\ell\RR^{n*}$ and $W_1=\Lambda^{\ell+1}\RR^{n*}$. The maps 
$$s_k:\Lambda^k\RR^{n*}\otimes\Lambda^{\ell+1}\RR^{n*}\to\Lambda^{k+1}\RR^{n*}\otimes\Lambda^\ell\RR^{n*}$$
are again induced by alternation in the first $k+1$ entries and thus equivariant.  

\subsection{Lie algebras and infinitesimal representations}
In the setting of Lie groups so in particular for matrix groups, finite dimensional representation theory is usually not done on the level of the group. In both settings, one can linearize the theory by passing to the Lie algebra, which looses almost no information and comes with a lot of technical simplification. Moreover, the constructions of equivariant maps that can be used in BGG theory actually comes from representations of Lie algebras, so we discuss the passage in some detail. In the setting of matrix groups, this just requires analysis on submanifolds of $\RR^N$ and actually is a beautiful application of differential calculus.  

We have already noted that a matrix group $G\subset GL(n,\RR)$ automatically is a smooth submanifold of $M_n(\RR)=\RR^{n^2}$, and since $G$ is a subgroup, it contains the unit matrix $\mathbb{I}$. The \textit{Lie algebra} $\mathfrak{g}$ of $G$ is then defined as the tangent space $T_{\mathbb{I}}G$ to $G$ at the unit matrix. General submanifold theory defines this tangent space via derivatives $c'(0)$ of local curves $c$ with $c(0)=\mathbb{I}$ that have values in $G$. A first important step is to prove that there is an alternative characterization of $\mathfrak{g}$ using the matrix exponential, which we write as $e^X\in GL(n,\RR)$ for $X\in M_n(\RR)$, namely
\[
\mathfrak{g} \coloneqq \mathrm{T}_{\mathbb{I}} G = \qty{X \in \mathrm{M}_n(\RR) : \ee^{tX} \in G,\; \forall t \in \RR}.
\]
For our main examples, the Lie algebras are
\[
\mathfrak{gl}(n, \RR) = \mathrm{M}_n(\RR), \quad \mathfrak{sl}(n, \RR) = \qty{X : \tr X = 0}, \quad \mathfrak{so}(n) = \qty{X : X^t = -X}.
\]
For $X\in M_n(\RR)$ and $A\in GL(n,\RR)$, one immediately verifies that $e^{AXA^{-1}}=Ae^XA^{-1}$. For a matrix group $G$ with Lie algebra $\mathfrak{g}$, this immediately implies that for $X\in\mathfrak{g}$ and $A\in G$, we get $AXA^{-1}\in\mathfrak{g}$.
The map $(A,X)\mapsto AXA^{-1}$ defines the adjoint action of $G$ on $\mathfrak g$, which defines a representation on the vector space $\mathfrak{g}$. Using $AXA^{-1}\in\mathfrak{g}$ for $A=e^{tY}$ with $Y\in\mathfrak{g}$ and differentiating at $t=0$, one concludes that for $X,Y\in\mathfrak{g}$, the commutator $[X,Y]:=XY-YX$ lies in $\mathfrak g$ and this defines the \textit{Lie bracket} $[\ ,\ ]:\mathfrak{g}\times\mathfrak{g}\to\mathfrak{g}$. This makes $\mathfrak{g}$ into a (finite-dimensional) Lie algebra (in the sense of algebra). 

Now suppose that $G, H$ are matrix groups with Lie algebras $\mathfrak{g}$ and $\mathfrak{h}$ and let $\varphi \colon G \to H$ be a smooth group homomorphism. Then of course $\varphi(\mathbb{I})=\mathbb{I}$ and hence the derivative of $\varphi$ in $\mathbb{I}$ is a linear map
\[
\varphi':=T_{\mathbb I}\phi \colon \mathfrak{g} \to \mathfrak{h}.
\]
Using the characterization of $e^X$ as a solution of a system of first order ODE, one concludes that for $X\in\mathfrak{g}$ one gets $\varphi(e^X)=e^{\varphi'(X)}$. This in turns easily implies that for $A\in G$, we get $\varphi'(AXA^{-1})=\varphi(A)\varphi'(X)\varphi(A)^{-1}$, which in turn shows that for $Y\in\mathfrak{g}$, we get 
$\varphi'([X, Y])=[\varphi'(X), \varphi'(Y)]$. Thus, $\varphi':\mathfrak{g}\to\mathfrak{h}$ is a homomorphism of Lie algebras, which often implies strong restrictions on the possible maps $\varphi'$. The equation $\varphi(e^X)=e^{\varphi'(X)}$ has another crucial consequence. The image $e^{\mathfrak{g}}$ of the matrix exponential contains an open neighborhood of $\mathbb{I}$ in $G$ and if $G$ is connected, it is easy to see that this neighborhood generates $G$ as a group. This implies that if $G$ is connected, then $\varphi$ is uniquely determined by $\varphi'$. 

\medskip 

Viewing a representation of $G$ on $V$ as a smooth homomorphism $\rho \colon G \to \mathrm{GL}(V)$ the derivative defines a Lie algebra homomorphism 
\[
\rho' \colon \mathfrak{g} \to \mathfrak{gl}(V) = L(V, V).
\]
As a linear map, $\rho'$ is equivalent to a bilinear map $\mathfrak{g} \times V \to V$, which we write as $(X, v) \mapsto X * v$ to distinguish from group actions. In this picture, the condition that $\rho'$ is a homomorphism of Lie algebras is equivalent to 
\[
[X, Y]* v = X* (Y*v) - Y*(X * v) \qquad \forall X,Y\in\mathfrak{g}.
\]
Notice that in this definition of a representation of a Lie algebra, no smoothness requirement occurs, this is pure linear algebra. The concepts of invariant subspaces, irreducibility, and equivariant maps can obviously be applie to Lie algebra representations.  

Of course, this implies that constructions for group representations lead to constructions for Lie algebra representations (which then easily extend to Lie algebra representations in the algebraic sense). One just has to be careful that the passage to the Lie algebra involves a differentiation, which changes the explicit formulae. The simplest way to compute infinitesimal representations is the observe that for $X\in\mathfrak g$ and $v\in V$, we get $X*v=\frac{d}{dt}|_{t=0}e^{tX}\cdot v$. For example for representations on $V$ and $W$ and the induced group representation on $L(V,W)$ we get $(e^{tX}\cdot f)(v)=e^{tX}\cdot f(e^{-tX}\cdot v)$ and differentiating this at $t=0$, we get 
$$
(X*f)(v)=X*(f(v))-f(X*v).  
$$
Similarly, for the tensor product $V\otimes W$, we get 
\[
X* (v \otimes w) = (X*v) \otimes w + v \otimes (X*w).
\]

There are several results that make sure that the representation theory of a connected matrix group $G$ is equivalent to the representation theory of its Lie algebra $\mathfrak{g}$. For our purposes the most important results are

\begin{thm*}
    Let $V$ be a representation of a connected group $G$. Then:

(1) $E \subset V$ is $G$-invariant $\Longleftrightarrow$ $E$ is $\mathfrak{g}$-invariant. In particular, $V$ is irreducible as a representation of $G$ if an only if it is irreducible as a representation of $\mathfrak{g}$. 

(2) A linear map $f \colon V_1 \to V_2$ is $G$-equivariant $\Longleftrightarrow$ $f$ is $\mathfrak{g}$-equivariant.
\end{thm*} 

This reduces the verification of irreducibility or of equivariance (a condition for all $g \in G$) to the verification of a linear condition (for all $X \in \mathfrak{g}$), which is typically much simpler.  

This result does not directly apply to the groups $GL(n,\RR)$ and $O(n)$ which each have two connected components. Still one can get a lot of information from the Lie algebra, and one usually just has to check how one element from the second connected components act to sort things out.  

\subsection{Remarks on the literature} 

As stated above, most texts on representation theory work mainly or exclusively in the setting of Lie algebra. A standard text in that direction is \cite{Fulton-Harris}, which also contains information on the relation between representation theory of $\mathfrak{gl}(n,\CC)$ and representation theory of the permutation group $\mathfrak{S}_n$, which both are often described in terms of Young diagrams. It should be pointed out, though, that the book starts with lots of examples, while the general theory as discussed above, is deferred to later parts of the book. Also, texts on Lie algebra representations usually contain quite a bit of material on the structure theory of semi-simple Lie algebras (roots, Dynkin diagrams, etc.) which are not really needed to deal with the concrete examples. Here the examples as treated in \cite{Fulton-Harris} are very useful. 

The basic Lie group / Lie algebra correspondence is discussed in most introductory texts on differential geometry, for example in Chapters 7 and 8 of \cite{Lee}, but this needs abstract smooth manifolds. While matrix groups provide beautiful examples and applications for the theory of submanifolds, there are only few textbooks in this setting, notable exceptions are the books \cite{Baker} and \cite{Tapp}. There are also several textbooks on Lie groups, for example \cite{Adams} which focuses on compact Lie groups and \cite{Knapp}, which also contains substantial amount of Lie algebra theory and preparation for the theory of infinite dimensional representations. So most of the material in that direction probably contains way too much material for a first introduction. 

\section{BGG diagrams from  representation theory}

The BGG constructions that are relevant for applied mathematics (so far) belong to two families. One of these families comes from projective differential geometry, the other from conformal differential geometry. We won't really go into these geometric aspect, for our purposes it suffices to know that for BGG sequences in dimension $n$, these are based on the two Lie groups $SL(n+1,\RR)$ and $O(n+1,1)$, respectively. Here we only discuss the examples related to $SL(n+1,\RR)$, the other one is discussed in my workshop talk \cite{Cap}.    

To set the stage, recall that we need 
\begin{enumerate}
    \item A representation $W$ of $SL(n,\RR)$ decomposed as $W=W_0\oplus\dots\oplus W_N$ into a direct sum of invariant subspaces. 
    \item $SL(n,\RR)$-equivariant maps $s_{\ell}:\RR^n\times W_{\ell}\to W_{\ell-1}$, such that $s_{\ell-1}(v_1,s_\ell (v_2,w))=s_{\ell-1}(v_2,s_{\ell}(v_1,w))$ for all $v_1,v_2\in\RR^n$ and $w\in W_\ell$. 
\item Viewing these maps as $s_{0,\ell}:W_\ell\to \RR^{n*}\otimes W_{\ell-1}$, we need equivariant extensions $s_{k,\ell}:\Lambda^k\RR^{n*}\otimes W_{\ell}\to\Lambda^{k+1}\RR^{n*}\otimes W_{\ell-1}$ such that $s_{k+1,\ell-1}\circ s_{k,\ell}=0$ for all $k,\ell$. 
\item We also need that for the corresponding maps $S_k$ on $W$-valued differential forms, $d+S_k$ corresponds to the covariant exterior derivative with respect to the linear connections $d+S_0$. 
\end{enumerate}
The commutativity condition in (2) makes sure that $d+S_0$ defines a flat connection on the trivial vector bundle $\RR^n\times W$. The requirement of equivariancy in (2) and (3) is motivated by the fact that one should view $W$ as inducing a natural vector bundle on $\RR^n$ rather than just being a vector space and this ties things in with differential forms. It is needed to get properties like affine invariance of the resulting sequences and it implies that things extend to manifolds without problems.

\subsection{Constructions of $W$ and the maps $s_{k,\ell}$ from representation theory}

The starting point for our considerations is to decompose $\RR^{n+1}=\RR\oplus\RR^n$ into the subspaces spanned by the first vector in the standard basis and by the remaining basis vectors, respectively. Correspondingly, we get an inclusion 
$G_0:=SL(n, \RR) \hookrightarrow SL(n+1, \RR)=:G$ defined by mapping a matrix $A$ to the block matrix $\begin{pmatrix}
    1 & 0\\ 0 & A\end{pmatrix}$. Looking at the Lie algebra $\mathfrak{g}=\mathfrak{sl}(n+1,\RR)$ decomposed into blocks in a similar way, we get 
\[
\mathfrak{sl}(n+1, \RR) = \qty{\mqty(-\tr(X) & \lambda \\ v & X) : v \in \RR^n,\; \lambda \in \RR^{n*},\; X \in \mathfrak{gl}(n, \RR)}.
\]
The Lie algebra of the subgroup $G_0$ in this picture corresponds to matrices with $v=0$, $\lambda=0$ and $\tr(X)=0$. Moreover, mapping $v\in\RR^n$ to the matrix $\left(\begin{smallmatrix} 0 & 0\\ v & 0 \end{smallmatrix}\right)$, we see that we can naturally view $\RR^n$ as a linear subspace of $\mathfrak{g}$. Observe that any two matrices of this form commute. To simplify notation, we will denote the block matrix in $SL(n+1,\RR)$ corresponding to $A\in SL(n,\RR)$ again by $A$ and the block matrix in $\mathfrak g$ corresponding to $v\in\RR^n$ again by $v$. In this language, the commutativity observed above reads as $[v_1,v_2]=0$ for any $v_1,v_2\in\RR^n$. Moreover, we can explicitly compute the adjoint action of $A$ on $v$, which we write as $\mathrm{Ad}(A)(v)$, via 
$$
\begin{pmatrix} 1 & 0\\ 0 & A \end{pmatrix}\begin{pmatrix} 0 & 0 \\ v & 0 \end{pmatrix}\begin{pmatrix} 1 & 0\\ 0 & A^{-1} \end{pmatrix}=\begin{pmatrix} 0 & 0 \\ Av & 0 \end{pmatrix}\begin{pmatrix} 1 & 0\\ 0 & A^{-1}\end{pmatrix}=\begin{pmatrix} 0 & 0 \\ Av & 0 \end{pmatrix},
$$
so $\mathrm{Ad}(A)(v)=Av$. 

Finally, we consider the so-called \textit{grading element}
\[
E = \mqty(\frac{1}{n+1} & 0 \\ 0 & -\frac{n}{n+1} \mathbb{I}) \in \mathfrak{g}.
\]
This acts diagonalizably on $\RR^{n+1}$ and hence on any representation of $\mathfrak{g}$ constructed from $\RR^{n+1}$. Alternatively, there are general results showing that $E$ acts diagonalizably on any finite-dimensional representation $W$ of $\mathfrak{g}$. On $\RR^n$, the two eigenvalues differ by $1$ and one proves that the situation is similar for any finite dimensional irreducible representation $W$: The eigenvalues are real numbers of the form $\lambda_0, \lambda_0 + 1, \lambda_0 + 2, \ldots, \lambda_0 + N$ for some $\lambda_0\in\RR$ and $N\in\mathbb{N}$. This provides a decomposition $W = W_0 \oplus W_1 \oplus \cdots \oplus W_N$. Similarly as above, one immediately verifies that for $A\in SL(n,\RR)$, we get $\mathrm{Ad}(A)(E)=E$. Finally for $v\in\RR^n$, we get $[E,v]=-v$. 

Now we only need one more piece of input on representations of matrix groups which easily follows from the definition of the infinitesimal representation. Take a representation of $G$ on $W$ denoted by $\cdot$ and denote the infinitesimal action by $*$, then for $\mathcal A\in G$ and $\mathcal X\in\mathfrak{g}$ we get $\textrm{Ad}(\mathcal{A})(\mathcal{X})=\mathcal{A}\mathcal{X}\mathcal{A}^{-1}$ and hence $e^{t\mathrm{Ad}(\mathcal{A})(\mathcal{X})}=\mathcal{A}e^{t\mathcal{X}}\mathcal{A}^{-1}$. Acting with this on $w\in W$ and differentiating at $t=0$, linearity of the action of $\mathcal A$ readily implies that 
\begin{equation}\label{equivar}
    \mathrm{Ad}(\mathcal A)(\mathcal{X})*w=\mathcal{A}\cdot (\mathcal{X}*(\mathcal{A}^{-1}\cdot w)).
\end{equation}  
Applying this in our setting to an irreducible representation $W=W_0\oplus\dots\oplus W_N$, $\mathcal{A}=A^{-1}\in G_0$ and $\mathcal X=E$, we get $E*w=A^{-1}\cdot (E*(A\cdot w)$ and hence $A\cdot (E*w)=E*(A\cdot w)$. This shows the the action of $A$ preserves each eigenspace of $E$, so each of the subspaces $W_\ell\subset W$ is $G_0$-invariant. 

Likewise $[E,v]=-v$ implies that $E*(v*w)-v*(E*w)=-v*w$. For $w\in W_\ell$, the second term in the left hand sice becomes $-(\lambda_0+\ell)v*w$ an bringing this to the other side shows that $v*w\in W_{\ell-1}$. Finally, applying \eqref{equivar} to $\mathcal A=A$ and $\mathcal X=v$, we get $Av*w=A\cdot (v*(A^{-1}\cdot w))$ or, equivalently, $Av*(A\cdot w)=A\cdot (v*w)$ which exactly means that $*:\RR^n\times W\to W$ is $SL(n,\RR)$-equivariant. So interpreting this as a map $s_0:W\to\RR^{n*}\otimes W$  this has all the properties we required above. 

\medskip

The extension to higher form degrees is provided by a general gadget from Lie algebra theory. Given a Lie algebra $\mathfrak g$ and a representation of $\mathfrak{g}$ on $V$, there is a general concept of the \textit{Lie algebra cohomology of $\mathfrak{g}$ with coefficients in $V$}, as well as a so-called \textit{standard complex} that computes this cohomology. The space in this standard complex are the spaces $\Lambda^k\mathfrak{g}^*\otimes V$ of $k$-linear, alternating maps $\mathfrak{g}^k\to V$ for $k=0,\dots,\dim(\mathfrak{g)}$, and the differential 
$$
\partial:\Lambda^k\mathfrak{g}^*\otimes V\to\Lambda^{k+1}\mathfrak{g}^*\otimes V 
$$
is defined by 
\begin{equation}\label{partial}
\begin{aligned}
    \partial\varphi(X_0,&\dots,X_k):=\textstyle\sum_{i=0}^k(-1)^iX_i*\varphi(X_0,\dots,\widehat{X_i},\dots,X_k)\\
    +&\textstyle\sum_{i<j}(-1)^{i+j}\varphi([X_i,X_j],X_0,\dots,\widehat{X_i},\dots,\widehat{X_j},\dots,X_k)
\end{aligned}
\end{equation}
for $\varphi:\mathfrak{g}^k\to V$ and $X_0,\dots,X_k\in\mathfrak{g}$ and with the hats denoting omission of arguments. It is then easy to prove in general that $\partial\circ\partial=0$, so one indeed obtains a complex, whose cohomology spaces are denoted by $H^*(\mathfrak{g},V)$. 

We apply this in a situation which is rather strange from the point of view of algebra, namely to the abelian Lie algebra $\RR^n$ (with trivial bracket) and the restriction to this subalgebra of $\mathfrak{sl}(n+1,\RR)$ of an irreducible representation $W$ of $\mathfrak{sl}(n+1,\RR)$. Hence in the notation introduced above, \eqref{partial} simplifies to 
$$
s_k(\varphi)(v_0,\dots,v_k)=\textstyle\sum_{i=0}^k(-1)^iv_i*\varphi(v_0,\dots,\widehat{v_i},\dots,v_k). 
$$
Equivariancy of the map $*:\RR^n\times W\to W$ easily implies that each of the maps $s_k$ is $SL(n,\RR)$-equivariant. Moreover, 
the general theory implies that $s_k\circ s_{k-1}=0$ and hence $\mathcal R(s_{k-1})\subset\mathcal{N}(s_k)\subset\Lambda^k\RR^{n*}\otimes W$ are $SL(n,\RR)$-invariant subspaces. In particular, each of the quotient spaces $H^k(\RR^m,W)$ naturally is a representation of $SL(n,\RR)$. It also follows from general results that for the induced maps $S_k$ on differential forms, $\dd+S_k$ is the covariant exterior derivative in degree $k$ induced by the connection $\DD=\dd+S_0$. Thus we have obtained all the properties required for the maps $s_k$ to induce nice $S$-operators.

\subsection{Examples}
Let us briefly outline how this construction leads to the examples discussed above. Recall the splitting $\RR^{n+1}=\RR\oplus\RR^n$ of the standard representation of $SL(n+1,\RR)$ with the components spanned by the elements $e_0$ and $e_1,\dots,e_n$ of the standard basis. The eigenvalues of $E$ on the summands are $\frac1{n+1}=\frac{-n}{n+1}+1$ and $\frac{-n}{n+1}$, respectively. Putting $W:=\RR^{(n+1)*}$, the corresponding eigenspaces are spanned by the elements $e^0$ and $e^1,\dots,e^{n+1}$ of the dual basis, with the negatives of the eigenvalues of the components in $\RR^{n+1}$, so $W=W_0\oplus W_1$ with $W_0=\RR$ and $W_1=\RR^{n*}$ as representations of $G_0=SL(n,\RR)$. The action of $v\in\RR^n\subset\mathfrak g$ on $(a,z)\in\RR^{n+1}$ by definition is given by $v*(a,z)=(0,av)$, so on the dual we get $v*\binom{t}{\lambda}=\binom{-\lambda(v)}{0}$. Interpreting this as a map $s_0:W_1\to \RR^{n*}\otimes W_0$, we get $s_0(\lambda)(v)=-\lambda(v)=v$, so $s_0=-id_{\RR^{n*}}$. For higher degrees, we by definition get 
$$
s_k(\varphi)(z_0,\dots,z_k)=\textstyle\sum_{i=0}^k(-1)^iz_i*\varphi(z_0,\dots,\widehat{z_i},\dots z_k)=-\textstyle\sum_{i=0}^k(-1)^i\varphi(z_0,\dots,\widehat{z_i},\dots z_k)(z_i). 
$$
Viewing $\varphi$ as a $k+1$-linear map which is alternating in the first $k$ entries, this is just $(-1)^{k+1}$ times the alternation of $\varphi$. Thus up to an overall minus sign (which does not change the outcome), this coincides with the maps inducing the twisted complex that leads to the Hessian complex as discussed in Section \ref{ex:Hess}. 

Starting with $W:=\Lambda^2\RR^{(n+1)*}$, we get a decomposition $W=W_0\oplus W_1$ with $W_0$ spanned by the vectors $e^0\wedge e^i$ with $i\geq 1$ and $W_1$ spanned by the vectors $e^i\wedge e^j$ with $1\leq i<j$, so $W_0\cong\RR^{n*}$ and $W_1\cong\Lambda^2\RR^{n*}$. The action of $\RR^n$ on $W$ is characterized by a Leibniz rule with respect to the wedge product, which shows that $v*(e^0\wedge e^i)=0$ and $v*(e^i\wedge e^j)=-e^i(v)e^j+e^j(v)e^i$. Since this is the negative of $e^i\wedge e^j$ applied to $V$, it corresponds to minus the inclusion as a map $s_0:\Lambda^2\RR^{n*}\to\RR^{n*}\otimes\RR^{n*}$. Similarly as discussed above, the induced maps $s_k:\Lambda^k\RR^{n*}\otimes\Lambda^2\RR^{n*}\to\Lambda^{k+1}\RR^{n*}\otimes\RR^{n*}$ are, up to a multiple, given by the alternation over the first $k+1$ entries, so we recover the maps leading to the elasticity complex, see Section \ref{ex:elast}.   

This extends to higher exterior powers $\Lambda^\ell\RR^{(n+1)*}$ which decompose as $\Lambda^{\ell-1}\RR^{n*}\oplus\Lambda^\ell\RR^{n*}$. Again the maps are induced by inclusions and alternations over the first $k+1$ arguments, viewed as a map $\Lambda^k\RR^{n*}\otimes\Lambda^\ell\RR^{n*}\to\Lambda^{k+1}\RR^{n*}\otimes\Lambda^{\ell-1}\RR^{n*}$. These lead to the BGG diagrams obtained from double forms.   

\subsection{Remarks on Kostant's theorem and relations to advanced representation theory}
Starting from $W$ and the maps $s_k$, the spaces occurring in the resulting BGG sequence on $U\subset\RR^n$ are the space $C^\infty(U,h_k)$, where $h_k=\mathcal{N}(s_k)\cap\mathcal{R}(s_{k-1})^\perp$, see Section \ref{BGG}. Now $h_k$ projects isomorphically onto the cohomology space $\mathcal{N}(s_k)/\mathcal{R}(s_{k-1})=H^k(\RR^n,W)$ and of course, the projection from $\mathcal{N}(s_k)$ to this quotient is $SL(n,\RR)$-equivariant. Hence we can understand the spaces $h_k$ by understanding the Lie algebra cohomology spaces $H^k(\RR^n,W)$.

It is important to point out here, that Lie algebra cohomology theory does not provide helpful input for this. On the one hand, algebraic interpretations of Lie algebra cohomologies exist only for low degrees and/or special representations. On the other hand, as we already noted above, the setup of the abelian Lie algebra $\RR^n$ acting on an irreducible representation $W$ of $SL(n+1,\RR)$ is rather bizarre from an algebraic point of view. Nonetheless, there is a result in representation theory which, as a special case, provides a complete, algorithmic description of the cohomology spaces $H^k(\RR^n,W)$ we need here. This is known as \textit{Kostant's version of the Bott-Borel-Weil theorem} from 1961, see \cite{Kostant}.  For a precise understanding, it is important that the representations we deal with actually extend to the subgroup $GL(n,\RR)\subset SL(n+1,\RR)$ corresponding to matrices of the form $\begin{pmatrix}\det(A)^{-1} & 0 \\ 0 & A \end{pmatrix}$. Already for the standard representation $\RR^{n+1}$ of $SL(n+1,\RR)$ this causes a subtle change compared to $SL(n+1,\RR)$ since the component spanned by the first basis vector now is not a trivial representation but corresponds to a density bundle. This helps in distinguishing representations in different degrees which are isomorphic as representations of the subgroup $SL(n,\RR)$. 

While the proof of Kostant's theorem requires only standard finite dimensional representation theory, even the formulation of the result lies quite a bit beyond the scope of these lectures. It needs the description of irreducible representations in terms of highest weights, the definition of the Weyl group of a simple Lie algebra (which is the permutation group $\mathfrak S_n$ in the case of $\mathfrak{sl}(n,\RR)$) and its action on weights. The general result states that each of the cohomology spaces is a direct sum of irreducible representations, which, even across different degrees, are non-isomorphic, and gives an algorithm for computing the highest weights of these irreducible representations in terms of the highest weight of the initial representation $W$. On this level, there are even computer implementations that compute the highest weights, but some understanding of representation theory is needed to see how these representations can be realized as subrepresentations of $\Lambda^k\RR^{n*}\otimes W$ for the appropriate $k$. Remarkably, the number of irreducible components that occur in $H^k(\RR^n,W)$ is independent of $W$, it is just the isomorphism types that change. Hence there is a uniform shape for all BGG sequences. It is a very strong an helpful part of the result that each of the irreducible components occurring in $H^k(\RR^n,W)$, occur (as representations of $GL(n,\RR)$) with multiplicity one in the representation $\Lambda^*\mathfrak{g}^*\otimes W$, so there is just one way how to realize them as subrepresentations. 

There are several simplifications of Kostant's general result in the case we need, which are mainly based on a better understanding of the Weyl group and the so called \textit{Hasse diagram} that describes the shape of the BGG sequences. In particular, in our case each of the representations $H^k(\RR^n,W)$ is irreducible, which in particular implies that $h_{k,\ell}\neq 0$ happens for exactly one $\ell$ in each degree $k$ and the representation $h_{k,\ell}$ is irreducible. In particular, $H^0(\RR^n,W)=W_0$ always holds and $W_0$ always is an irreducible representation of $SL(n+1,\RR)$. Also the orders of the operators in the BGG sequence can be easily determined from representation theory data without going into any details of the construction. Finally, it turns out that irreducible representations of $SL(n+1,\RR)$ can be parametrized by pairs $(W_0,r)$, where $W_0$ is an irreducible representation of $SL(n,\RR)$ and $r\geq 0$ is an integer. The representation $W$ corresponding to the pair $(W_0,r)$ then has the property that $W_0$ is the first space in its decomposition, so the first operator in the BGG sequence is defined on $C^\infty(U,W_0)$ and that the first BGG operator has order $r$, see \cite{BCEG}. In particular, taking from $W_0:=\RR$ we obtain the de Rham complex for $r=0$ and the Hessian complex for $r=1$ and there are complexes for $r\geq 2$, which have the $(r+1)$-st derivative of a function as their first operator. It turns out that for these examples $W=S^r\RR^{(n+1)*}$, they are discussed in Section 2.4 of \cite{Cap-Hu}. 

At this point it may be rather surprising why Kostant was interested in determining the cohomology spaces $H^*(\RR^n,W)$ considered above. The reason lies in another unexpected relation of these cohomologies to geometric or analytic problems, which requires passing to complex Lie groups and Lie algebras. The example relevant for our case is the group $G:=SL(n+1,\CC)$ which obviously acts transitively on the space $\CC P^n$ of complex lines in $\CC^{n+1}$. This is an example of a compact homogeneous space of a complex simple Lie group, which generalizes to the remarkable class of so-called \textit{generalized flag manifolds}. The stabilizer $P$ of the line spanned by the first basis vector in $\CC^{n+1}$ is the group of block-upper-triangular matrices of the form $\begin{pmatrix} \det(A)^{-1} & Z \\ 0 & A\end{pmatrix}$ with $A\in GL(n,\CC)$ and $Z\in\CC^{n*}$. This is a simple example of a \textit{parabolic subgroup} in a simple Lie group and $\CC P^n\cong G/P$. There is a standard construction which associates to a representation of $P$ on a vector space $E$ a vector bundle $G\times_PE\to G/P$ with typical fiber $E$ together with an action of $G$ on the total space that lifts the natural action on $G/P$ (a \textit{homogeneous vector bundle}).

For the group $P$, complete reducibility does not hold, but still one can consider the case that $E$ is irreducible. It turns out that such representations are exactly irreducible representations of the block diagonal part with trivial action of the $Z$-part. For the complex Lie group $GL(n,\CC)$, there is the concept of holomorphic representations on complex vector spaces, and these correspond to real representations of $GL(n,\RR)$. If $E$ is a holomorphic representation, then $G\times_PE$ is a holomorphic vector bundle over a complex manifold, so the concept of holomorphic sections makes sense. Moreover, spaces of sections of a homogeneous vector bundle automatically carry a representation of the group $G$. Compactness of $G/P$ easily implies that the space of global holomorphic sections of $G\times_PE$ is finite dimensional, so one obtains a finite dimensional representation of $G$ in this way, which in a special case is described by the \textit{Borel-Weil theorem}. This result was extended by R.\ Bott to a description of the sheaf cohomology of the sheaf of local holomorphic sections of the homogeneous vector bundle $E$ (which may be non-trivial in cases where there are no global holomorphic sections). This is referred to as the \textit{Bott-Borel-Weil theorem}. 

Kostant's theorem was motivated by finding an alternative proof for this result. The basic idea for this is that there is a general approach to computing this type of sheaf cohomology via Dolbeault cohomology. There is a standard complex for computing this sheaf cohomology starting with smooth sections of the bundles and continuing with so-called $(0,k)$-forms with values in $G\times_PE$ and the Dolbeault differential is a component of the exterior derivative. Using a some representation theory of compact groups (the so-called Peter-Weyl theorem), the Bott-Borel-Weil theorem can be deduced (in our case) from the knowledge of $H^*(\RR^n,E)$. 

The BGG construction outlined in this course can be adapted to the case of generalized flag manifolds. This on the one hand connects to the original work of Bernstein-Gelfand-Gelfand, which led to the name of the construction, on the other hand it is the basis for the BGG construction in the setting of parabolic geometries, which can be viewed as ``curved analogs'' of generalized flag manifolds. Given an irreducible representation $W$ of $G=SL(n+1,\RR)$ with the decomposition $W=W_0\oplus \dots\oplus W_N$ as discussed above, the subspace $W_\ell$ are not $P$ invariant, but the action of $P$ maps each $W_\ell$ to $\oplus_{j\geq\ell}W_j$. Consequently, each of the subspaces $W^i:=\oplus_{j\geq i}W_j$ is $P$-invariant and we can view $W_0$ more conceptually, as $W/W^1$. It is rather easy to show that the homogeneous vector bundle $G\times_PW$ can be canonically trivialized, which in turn defines a flat connection $\nabla$ on this bundle. (This is the model for tractor bundles, and this connection has a tensorial part related to a Lie algebra cohomology differential built into its construction.) This flat connection can be used to define a twisted de Rham complex of differential forms with values in $G\times_PW$. Now there is a construction of analogs of the maps $t$ used in our construction based on representation theory. This uses a Lie algebra homology differential which is dual to the Lie algebra cohomology differential we discussed. This differential is even $P$-equivariant and thus provides tensorial operators 
$$
\partial^*:\Omega^k(G/P,G\times_PW)\to\Omega^{k-1}(G/P,G\times_PW).
$$
One can run the BGG construction based on these operations,  which leads to higher order operators acting on the bundles $\mathcal{N}(\partial^*)/\mathcal{R}(\partial^*)$. These bundles are associated to Lie algebra homology spaces that agree with the cohomology groups computed by Kostant's theorem. In particular, in degree zero, we obtain $G\times_PW_0$ as the homology bundle, and the BGG complex computes the same cohomology as the twisted de Rham complex. It should be remarked here, that the global topology of generalized flag manifolds (including their cohomology groups) is rather complicated and can be best understood using representation theory. 

The construction implies that the cohomology bundles are homogeneous vector bundles and hence the spaces of sections of these bundles form infinite dimensional representations of $G$ (which can be complete to representations on Hilbert spaces if needed). Naturality of the BGG construction implies that the BGG operators are equivariant and thus define intertwining operators between these infinite dimensional representations, which provides another connection to representation theory. While the study of  infinite dimensional representations and intertwining operators between them usually needs a lot of functional analysis and intertwiners are non-local operators in general, there is a different approach to the study of intertwining differential operators which a local operators. Via a duality, they are related to homomorphisms of so-called \textit{generalized Verma modules}. These form a class of infinite dimensional representations that can be studied by purely algebraic methods. Without going deeply into this theory, one obtains strong restrictions on the existence of intertwining operators, which together with Kostant's theorem imply that most of them occur in BGG sequences. This is relevant for differential geometry, since for some of the generalized flag manifolds, the group $G$ is the automorphism group of a geometric structure on $G/P$, so one e.g.\ obtains restrictions on existence of conformally invariant differential operators.  

The relation to representation theory goes a bit further, namely to the results of Berstein-Gelfand-Gelfand in \cites{BGG1,BGG2} and their generalization to J.\ Lepowsky in \cite{Lepowsky}. These are existence results on homomorphisms between Verma modules (which are less relevant to geometry) respectively generalized Verma modules. In particular, they show (in the respective setting), that this leads to a resolution of any finite dimensional irreducible representation $W$ of $G$ by homomorphisms of (generalized) Verma modules, which are dual to the BGG complexes of differential operators as constructed above. While there are some applications of these results to differential geometry, they are not needed in the construction of the differential operator version of the resolutions or its generalization to curved geometries.


\begin{bibdiv}
    \begin{biblist}

    \bib{Adams}{book}{
   author={Adams, J. Frank},
   title={Lectures on Lie groups},
   publisher={W. A. Benjamin, Inc., New York-Amsterdam},
   date={1969},
   pages={xii+182},
   review={\MR{0252560}},
}

\bib{Baker}{book}{
   author={Baker, Andrew},
   title={Matrix groups},
   series={Springer Undergraduate Mathematics Series},
   note={An introduction to Lie group theory},
   publisher={Springer-Verlag London, Ltd., London},
   date={2002},
   pages={xii+330},
   isbn={1-85233-470-3},
   review={\MR{1869885}},
   doi={10.1007/978-1-4471-0183-3},
}

\bib{BGG1}{article}{
   author={Bern\v ste\u in, I. N.},
   author={Gel\cprime fand, I. M.},
   author={Gel\cprime fand, S. I.},
   title={Structure of representations that are generated by vectors of
   highest weight},
   language={Russian},
   journal={Funkcional. Anal. i Prilo\v zen.},
   volume={5},
   date={1971},
   number={1},
   pages={1--9},
   issn={0374-1990},
   review={\MR{0291204}},
}

\bib{BGG2}{article}{
   author={Bern\v ste\u in, I. N.},
   author={Gel\cprime fand, I. M.},
   author={Gel\cprime fand, S. I.},
   title={Differential operators on the base affine space and a study of
   ${\germ g}$-modules},
   conference={
      title={Lie groups and their representations},
      address={Proc. Summer School, Bolyai J\'anos Math. Soc., Budapest},
      date={1971},
   },
   book={
      publisher={Halsted Press, New York-Toronto, Ont.},
   },
   date={1975},
   pages={21--64},
   review={\MR{0578996}},
}

\bib{BCEG}{article}{
   author={Branson, Thomas},
   author={\v Cap, Andreas},
   author={Eastwood, Michael},
   author={Gover, A. Rod},
   title={Prolongations of geometric overdetermined systems},
   journal={Internat. J. Math.},
   volume={17},
   date={2006},
   number={6},
   pages={641--664},
   issn={0129-167X},
   review={\MR{2246885}},
   doi={10.1142/S0129167X06003655},
}

    \bib{Cap-Hu}{article}{
   author={\v Cap, Andreas},
   author={Hu, Kaibo},
   title={BGG sequences with weak regularity and applications},
   journal={Found. Comput. Math.},
   volume={24},
   date={2024},
   number={4},
   pages={1145--1184},
   issn={1615-3375},
   review={\MR{4783639}},
   doi={10.1007/s10208-023-09608-9},
}

\bib{Cap}{misc}{
 author={\v Cap, Andreas},
 title={A Riemannian version of the BGG construction},
 note={Talk at the Workshop ``Theory of Differential Complexes and Related Models'' in the ESI program, video available via https://www.youtube.com/watch?v=SB3--rak290 , slides available as https://www.mat.univie.ac.at/~cap/files/ESI26-beamer.pdf . }
}
    
\bib{Fulton-Harris}{book}{
   author={Fulton, William},
   author={Harris, Joe},
   title={Representation theory},
   series={Graduate Texts in Mathematics},
   volume={129},
   note={A first course;
   Readings in Mathematics},
   publisher={Springer-Verlag, New York},
   date={1991},
   pages={xvi+551},
   isbn={0-387-97527-6},
   isbn={0-387-97495-4},
   review={\MR{1153249}},
   doi={10.1007/978-1-4612-0979-9},
}

\bib{Knapp}{book}{
   author={Knapp, Anthony W.},
   title={Lie groups beyond an introduction},
   series={Progress in Mathematics},
   volume={140},
   edition={2},
   publisher={Birkh\"auser Boston, Inc., Boston, MA},
   date={2002},
   pages={xviii+812},
   isbn={0-8176-4259-5},
   review={\MR{1920389}},
}

\bib{KN}{book}{
   author={Kobayashi, Shoshichi},
   author={Nomizu, Katsumi},
   title={Foundations of differential geometry. Vol I},
   publisher={Interscience Publishers (a division of John Wiley \& Sons,
   Inc.), New York-London},
   date={1963},
   pages={xi+329},
   review={\MR{0152974}},
}

\bib{Kostant}{article}{
   author={Kostant, Bertram},
   title={Lie algebra cohomology and the generalized Borel-Weil theorem},
   journal={Ann. of Math. (2)},
   volume={74},
   date={1961},
   pages={329--387},
   issn={0003-486X},
   review={\MR{0142696}},
   doi={10.2307/1970237},
}

  \bib{Lee}{book}{
   author={Lee, John M.},
   title={Introduction to smooth manifolds},
   series={Graduate Texts in Mathematics},
   volume={218},
   edition={2},
   publisher={Springer, New York},
   date={2013},
   pages={xvi+708},
   isbn={978-1-4419-9981-8},
   review={\MR{2954043}},
}

\bib{Lepowsky}{article}{
   author={Lepowsky, J.},
   title={A generalization of the Bernstein-Gelfand-Gelfand resolution},
   journal={J. Algebra},
   volume={49},
   date={1977},
   number={2},
   pages={496--511},
   issn={0021-8693},
   review={\MR{0476813}},
   doi={10.1016/0021-8693(77)90254-X},
}

\bib{Tapp}{book}{
   author={Tapp, Kristopher},
   title={Matrix groups for undergraduates},
   series={Student Mathematical Library},
   volume={79},
   edition={2},
   publisher={American Mathematical Society, Providence, RI},
   date={2016},
   pages={viii+239},
   isbn={978-1-4704-2722-1},
   review={\MR{3468869}},
   doi={10.1090/stml/079},
}

    \end{biblist}
\end{bibdiv}
\end{document}